\documentclass{article}
\usepackage[utf8]{inputenc}
\usepackage{authblk}
\usepackage{setspace}
\usepackage[margin=1.25in]{geometry}
\usepackage{graphicx}
\usepackage{caption}
\usepackage{amsmath}
\usepackage{amssymb}
\usepackage{amsthm}
\usepackage{bm}
\usepackage{hyperref}
\hypersetup{colorlinks=true,linkcolor=blue,citecolor=blue,urlcolor=blue}
\usepackage{booktabs}
\usepackage{array}
\usepackage{tabularx}
\usepackage{threeparttable}
\usepackage{longtable}
\usepackage{needspace}
\usepackage[section]{placeins}
\usepackage{enumitem}
\setlist{itemsep=3pt,parsep=1pt}
\usepackage{microtype}

\newtheorem{proposition}{Proposition}
\newtheorem{lemma}{Lemma}
\newtheorem{definition}{Definition}
\newtheorem{corollary}{Corollary}
\newtheorem{remark}{Remark}
\newtheorem{assumption}{Assumption}
\newcounter{algorithm}

\usepackage[numbers,sort&compress]{natbib}

\newcommand{\E}{\mathbb{E}}
\newcommand{\R}{\mathbb{R}}
\newcommand{\Mcal}{\mathcal{M}}
\newcommand{\Acal}{\mathcal{A}}
\newcommand{\Bcal}{\mathcal{B}}
\newcommand{\Fcal}{\mathcal{F}}
\newcommand{\fext}{\mathbf{f}^{\mathrm{ext}}}
\newcommand{\fbody}{\mathbf{f}^{\mathrm{body}}}

\title{Identifiable Uniqueness of Impulsively Forced, Inert, and Stabilized Body Trajectories}

\author[1]{Bo Pieter Johannes Andr\'ee\thanks{The findings, interpretations, and conclusions expressed herein are entirely those of the author and do not necessarily represent the views of the International Bank for Reconstruction and Development/World Bank, its Board of Executive Directors, or the governments they represent. No World Bank resources were used to conduct this research. This research received no specific funding. The author declares no competing interests. Contact: bandree(at)worldbank.org}}
\affil[1]{Data Group and Multilateral \& UN, World Bank, Geneva, Switzerland.}
\date{September 21, 2026}

\newcommand{\mcNrep}{1000}
\newcommand{\mcHorizon}{15}
\newcommand{\mcKhi}{0.85}
\newcommand{\mcKlo}{0.50}
\newcommand{\mcOutliers}{3--10}
\newcommand{\mcLambda}{0.85}
\newcommand{\mcLambdaLo}{0.50}
\newcommand{\mcGrad}{4}
\newcommand{\mcSdNu}{0.3}
\newcommand{\mcTailDf}{5}
\newcommand{\mcAlpha}{5}
\newcommand{\mcSnrRatioLo}{1}
\newcommand{\mcSnrRatioHi}{3}
\newcommand{\mcXsecN}{30}
\newcommand{\mcTshort}{60}
\newcommand{\mcSizeShort}{0.08}
\newcommand{\mcPowHiShort}{0.90}
\newcommand{\mcPowLoShort}{0.54}
\newcommand{\mcPowSnrOneShort}{0.29}
\newcommand{\mcPowSnrThreeShort}{0.16}
\newcommand{\mcLeakShort}{0.43}
\newcommand{\mcCleanShort}{0.07}
\newcommand{\mcKurtShort}{-1.1}
\newcommand{\mcKurtSEShort}{0.06}
\newcommand{\mcKurtPosShort}{0.25}
\newcommand{\mcKurtLamShort}{-0.1}
\newcommand{\mcWCShort}{+2.9}
\newcommand{\mcGradNaiveShort}{0.27}
\newcommand{\mcDetNaiveShort}{0.07}
\newcommand{\mcDetFullShort}{0.04}
\newcommand{\mcGradStdShort}{0.07}
\newcommand{\mcGradFitShort}{0.06}
\newcommand{\mcLeakRobShort}{0.11}
\newcommand{\mcDispSizeShort}{0.05}
\newcommand{\mcXsecSizeShort}{0.06}
\newcommand{\mcXsecPowHiOneShort}{0.97}
\newcommand{\mcXsecPowHiThreeShort}{0.42}
\newcommand{\mcXsecPowLoOneShort}{0.77}
\newcommand{\mcXsecPowLoThreeShort}{0.29}
\newcommand{\mcTmed}{120}
\newcommand{\mcSizeMed}{0.05}
\newcommand{\mcPowHiMed}{0.90}
\newcommand{\mcPowLoMed}{0.55}
\newcommand{\mcPowSnrOneMed}{0.31}
\newcommand{\mcPowSnrThreeMed}{0.14}
\newcommand{\mcLeakMed}{0.37}
\newcommand{\mcCleanMed}{0.05}
\newcommand{\mcKurtMed}{-1.7}
\newcommand{\mcKurtSEMed}{0.08}
\newcommand{\mcKurtPosMed}{0.16}
\newcommand{\mcKurtLamMed}{-0.2}
\newcommand{\mcWCMed}{+4.2}
\newcommand{\mcGradNaiveMed}{0.24}
\newcommand{\mcDetNaiveMed}{0.05}
\newcommand{\mcDetFullMed}{0.05}
\newcommand{\mcGradStdMed}{0.07}
\newcommand{\mcGradFitMed}{0.07}
\newcommand{\mcLeakRobMed}{0.09}
\newcommand{\mcDispSizeMed}{0.06}
\newcommand{\mcXsecSizeMed}{0.06}
\newcommand{\mcXsecPowHiOneMed}{1.00}
\newcommand{\mcXsecPowHiThreeMed}{0.64}
\newcommand{\mcXsecPowLoOneMed}{0.97}
\newcommand{\mcXsecPowLoThreeMed}{0.42}
\newcommand{\mcTlong}{250}
\newcommand{\mcSizeLong}{0.06}
\newcommand{\mcPowHiLong}{0.91}
\newcommand{\mcPowLoLong}{0.57}
\newcommand{\mcPowSnrOneLong}{0.28}
\newcommand{\mcPowSnrThreeLong}{0.13}
\newcommand{\mcLeakLong}{0.39}
\newcommand{\mcCleanLong}{0.05}
\newcommand{\mcKurtLong}{-2.4}
\newcommand{\mcKurtSELong}{0.10}
\newcommand{\mcKurtPosLong}{0.07}
\newcommand{\mcKurtLamLong}{-0.4}
\newcommand{\mcWCLong}{+4.8}
\newcommand{\mcGradNaiveLong}{0.17}
\newcommand{\mcDetNaiveLong}{0.06}
\newcommand{\mcDetFullLong}{0.06}
\newcommand{\mcGradStdLong}{0.06}
\newcommand{\mcGradFitLong}{0.04}
\newcommand{\mcLeakRobLong}{0.10}
\newcommand{\mcDispSizeLong}{0.06}
\newcommand{\mcXsecSizeLong}{0.06}
\newcommand{\mcXsecPowHiOneLong}{1.00}
\newcommand{\mcXsecPowHiThreeLong}{0.91}
\newcommand{\mcXsecPowLoOneLong}{1.00}
\newcommand{\mcXsecPowLoThreeLong}{0.72}
\newcommand{\mcGradNaiveRange}{0.17--0.27}

\newcommand{\figEta}{0.85}
\newcommand{\figKmax}{0.6}
\newcommand{\figTau}{10}
\newcommand{\figLambda}{0.7}
\newcommand{\chiRateAU}{2.0\times10^{-7}}
\newcommand{\chiDaysOneAU}{58}
\newcommand{\chiDaysOneFiveAU}{107}

\begin{document}
\maketitle

\begin{abstract}
An unusual object in the sky raises two related questions: what kind of object it is and what physical process drives its motion. Departures from a gravity-only trajectory can provide evidence, but the same apparent acceleration may be generated by distinct dynamical mechanisms. We formulate this as an identification problem and establish conditions for identifiable uniqueness among three reference classes: inert dynamics, impulsive forcing, and active stabilization such as a controlled spacecraft. The distinction rests on how each class responds to disturbances. Although different mechanisms can produce similar average trajectories, the specified reference models imply different stochastic responses around those trajectories. We show when these differences uniquely determine the dynamical class from the observed trajectory and when several classes remain observationally compatible.

\end{abstract}


\vfill
\pagebreak

\section{Introduction}
\label{sec:intro}

The discovery of an unusual object in the sky prompts efforts to determine what it is and what
drives its motion. Astronomers compare its observed trajectory with a gravitational prediction to
detect non-gravitational acceleration. Statistically significant departures have revealed
photometrically inactive objects known as dark comets
\citep{seligman2023darkcomets,seligman2024populations}. Other gravitational searches select
interstellar-meteor candidates by excess speed relative to solar escape \citep{cloete2026meteors}
or seek dark objects through their perturbations of detector test masses \citep{thoss2025dark}.
Several physical processes can explain a detected non-gravitational acceleration,
particularly over short observational arcs \citep{micheli2018oumuamua,jewitt2022interstellar};
solar radiation pressure on a low-density body is one such competing account of the
\mbox{`Oumuamua} anomaly \citep{bialy2018radiation}.
Maneuver-detection methods likewise infer departures from specified dynamics
\citep{willsky1976failure,holzinger2012control,pastor2022optical,porcelli2022radar}.
Classifying the object requires features of its motion that distinguish the competing
physical explanations.

The discussion surrounding 1998~KY$_{26}$ illustrates the distinction between anomaly detection
and physical identification. The object has been associated with the dark-comet population despite
the absence of detected dust \citep{santanaros2025ky26}. One analysis explains its astrometric record
through solar radiation pressure and thermal emission \citep{farnocchia2025ky26}; a preprint instead
constructs a two-impulse path connecting its orbital history to the lost Phobos~1 spacecraft
\citep{classify2606}. The planned Hayabusa2 rendezvous offers an opportunity for direct characterization
\citep{hirabayashi2021hayabusa2}. For objects accessible only through remote observations,
the question remains how far trajectory information can distinguish these physical explanations.

A trajectory residual is the difference between an object's position and a specified reference
path, such as a gravity-only orbit or a forecast based on earlier observations. Such departures
can leave the physical class unresolved for several reasons. First, distinct mechanisms
may generate similar residuals. A small residual can arise because the forcing
itself is weak, or because a larger disturbance is offset by corrective action. Second, the
informativeness of a diagnostic depends on the disturbance history. Changes in variance or
distributional shape between quiet and disturbed periods provide little separation when the
forcing environment changes too little over the observed arc. Third, there is an exact limiting
ambiguity: vanishing forcing and perfect cancellation produce the same zero residual under the
observation model considered here. Away from that limit, mechanisms with similar mean trajectories
can still differ in their stochastic responses to disturbances.

We derive conditions for distinguishing three reference regimes motivated by natural forcing
\citep{andree2026stability} and actuator behavior \citep{andree2026thruster}:
inert propagation without body-generated pulses or correction at the sampling scale (I),
impulsively forced motion (C), and stabilized motion (S).
Spacecraft guidance and disturbance-rejection systems provide the physical setting for corrective
action \citep{dennehy2023gnc,giulicchi2013lisa}.
Artificial objects make the distinction between inert propagation and stabilization concrete.
The rocket booster 2020~SO, recognized from its astrometric trajectory, and the Tesla Roadster
launched into solar orbit are inert ballistic bodies \citep{petrova2021so,rein2018roadster},
while an autonomously navigated spacecraft such as Deep~Space~1 exemplifies stabilized motion
\citep{bhaskaran2000ds1}.

The analysis compares the response of residual motion to disturbances rather than relying on the
mean trajectory alone. Errors in predicting the next observation supply one diagnostic. Their
stationarity---a stable mean and lag-covariance structure over time---is a shared regularity
condition for the inert and stabilized reference models. A separate restriction on scalar excess
kurtosis, a measure of distributional shape based on second and fourth moments, permits the
exclusion of impulsive forcing under the stated comparison laws. This first step can leave both
inert and stabilized motion compatible. To distinguish that pair, we compare the response with
the same forcing propagated without correction. Under the stated correction and noise conditions,
a smaller event-related variance increase or lower conditional excess kurtosis separates
stabilization from inert propagation. When these moment comparisons are uninformative, an isolated
disturbance supplies a complementary comparison. With background forcing and sensing noise held
fixed, a strengthening correction response reduces the conditional variance of the unpredictable
part of subsequent velocity increments, whereas the inert reference produces no corresponding
variance response.

A signature collects statistical features of the observed motion, such as variance, excess kurtosis,
and responses to disturbances. These signatures express the classical identification problem in
terms of physical regimes \citep{rothenberg1971identification}. Each regime generates a set of
signatures as its forcing, correction, and observation parameters vary. The physical class is
identifiable from these features when the true signature belongs to only one regime set, even if
several parameter values within that class produce it. For the nearest-signature classifier,
consistent signature estimation and a positive distance from the true signature to every competing
regime set give consistent recovery of the class. This gap supplies the separation condition in
standard extremum-estimator consistency arguments \citep[Theorem~2.1]{newey1994large}.
When several regimes generate the same signature, the selected features identify their compatible
class set. With a positive gap to excluded regimes, consistent estimation makes the selected class
belong to that set with probability tending to one. This permits nonuniqueness, as in consistency
arguments allowing a set of population maximizers \citep{andree2020theory}.

Finite measurement precision adds a second requirement. Suppose a calibrated error radius
bounds the distance between the estimated and true signatures, incorporating sampling variation,
instrumental uncertainty, and uncertainty in the reference comparisons. When the same radius is
used to determine compatibility with a regime set, a population separation larger than twice that
radius is sufficient to exclude every competing regime. One radius allows for error in estimating
the true signature; the other allows for the tolerance in matching the estimate to a candidate
regime. A separate residual-amplitude requirement excludes the observational limit in which
near-perfect correction is indistinguishable from vanishing forcing. Population separability
therefore concerns whether the regimes imply different observable signatures, while observational
resolvability concerns whether the available data distinguish them at the achieved precision.

Additional classes or finer subgroups enter through their signature sets; each addition changes
the rivals against which uniqueness must be established. The distance criterion provides a basis
for classification algorithms, illustrated here by a rule that retains every regime compatible
with the estimated signature at the calibrated tolerance. Numerical experiments illustrate the
diagnostic comparisons under specified forcing and observation laws. Unresolved contrasts guide
further trajectory measurements or independent observations of activity and composition that
could distinguish the remaining explanations.

\section{Theoretical framework}
\label{sec:framework}

A body's departure from a gravity-only path records the combined effects of disturbances and
any action that opposes them. A small departure can reflect weak forcing or effective correction;
a large one can persist after an impulse even when subsequent disturbances are being suppressed.
Distinguishing these mechanisms requires separating the inputs that change velocity, their
accumulation into displacement, and the errors introduced when that displacement is observed.
The target is the physical regime, so different forcing amplitudes or correction parameters may
remain compatible within an identified class. Appendix~\ref{app:notation} lists the recurring
notation.

\subsection{Identification of the physical regime}
\label{sec:inference}

An astronomer follows one object through repeated measurements of its position on the sky.
Optical astrometry records angular positions against a stellar reference frame; orbit determination
compares those measurements with predicted positions and accounts for their uncertainties
\citep{carpino2003errors,veres2017statistical}. The tracked object's motion may be inert (I), impulsively forced (C), or stabilized (S). The data are the measured
departures from a gravity-only reference, not direct observations of the forces or corrective commands.

Let $\mathbf y_t$ denote this observed residual at epoch $t$, expressed in the instrument's
coordinates, and let $\mathcal Y_T=(\mathbf y_1,\ldots,\mathbf y_T)$ be the record from $T$ epochs.
A classification rule assigns a regime $\widehat R_T\in\mathcal R=\{I,C,S\}$ to that record.
Within each candidate regime, $\theta\in\Theta_R$ describes the forcing and correction parameters
and any unknown quantities in the reference and measurement model. For the true regime $R_0$
and specification $\theta_0$, consistency means
\begin{equation}
\Pr_{R_0,\theta_0}(\widehat R_T=R_0)\longrightarrow1\qquad\text{as }T\to\infty.
\label{eq:class-consistency-target}
\end{equation}
This consistency target concerns the physical regime; individual forcing amplitudes and
correction parameters within that regime need not be uniquely identified.

\citet{rothenberg1971identification} defines identification through the probability laws of
observable data. In this setting, $P^{(T)}_{R,\theta}$ is the law of the possible astrometric records
under a candidate physical explanation, with the observing times, geometry, and instrument model
specified. If a rival regime $R\ne R_0$ admits a specification $\theta$ such that
\[
P^{(T)}_{R_0,\theta_0}=P^{(T)}_{R,\theta}
\qquad\text{for every }T,
\]
no classifier can consistently distinguish those two explanations. Every decision based on the
record has the same distribution under both. More observations cannot resolve an ambiguity that
persists in the entire observation law.

An extremum estimator compares candidate mechanisms and their nuisance parameters by maximizing
a criterion $Q_T(R,\theta)$ computed from the record. Its population counterpart $Q(R,\theta)$
is the limiting criterion under the actual data-generating process. The consistency argument in
\citet[Section~2.2, Assumption~6 and Theorem~7]{andree2020theory} combines uniform convergence
with identifiable uniqueness of the maximum; \citet[Theorem~2.1]{newey1994large} gives the
compact-parameter-space version. For regime identification, several maximizing nuisance
specifications may belong to the same physical class. Write
\[
\Theta=\bigcup_{R\in\mathcal R}\bigl(\{R\}\times\Theta_R\bigr),\qquad
(\widehat R_T,\widehat\theta_T)
\in\arg\max_{(R,\theta)\in\Theta}Q_T(R,\theta).
\]
Suppose $\Theta$ is compact, with a discrete metric on the regime index, $Q$ is continuous,
and $Q_T$ converges uniformly in probability to $Q$. Let $Q$ be maximized at the true
specification $(R_0,\theta_0)$, and define the set of population maximizers
\begin{equation}
\Theta^\star=\arg\max_{(R,\theta)\in\Theta}Q(R,\theta).
\label{eq:population-solution-set}
\end{equation}
Under these conditions the estimator's distance to $\Theta^\star$ converges to zero in probability:
a unique maximizer yields point consistency, and several maximizers yield convergence to their set.

For a correctly specified nonlinear spatial time-series model,
\citet[Section~4.3.3, Theorem~12]{andree2020theory} proves almost-sure convergence of the distance
from the maximum-likelihood estimate to the parameter set reproducing the true conditional
distribution. Several parameter values in an overspecified nonlinear model can represent the same
linear data-generating process. Convergence to a set of population maximizers can still identify one physical class
when all maximizing forcing and correction specifications belong to that class. Define
\begin{equation}
\mathcal R^\star=
\left\{R\in\mathcal R:
\exists\theta\in\Theta_R\text{ such that }(R,\theta)\in\Theta^\star\right\}.
\label{eq:identified-regime-set}
\end{equation}
If $\mathcal R^\star=\{R_0\}$, the positive population gap between $R_0$ and the other regime
components gives class consistency under the conditions above. If $\mathcal R^\star=\{I,S\}$,
the estimated class belongs to the inert--stabilized pair with probability tending to one;
the convergence result does not select a unique member.

The unresolved question is which features prevent one mechanism from reproducing another's
motion. Allowing arbitrary forcing and arbitrary correction would leave many such substitutions
possible; restrictions on their response to disturbances are therefore essential. Related
identification arguments for score-driven time-series models obtain a unique population-likelihood
maximum from restrictions on the recursion and innovation law \citep{blasques2022maximum}.
Here the required distinction is between physical classes. Rather than identify every coefficient,
we seek features of the measured trajectory that rival classes cannot reproduce. Let
$s(R,\theta)\in\mathbb R^m$ collect such features, including variance, excess kurtosis, and
responses to disturbances, after propagation into the measured coordinates. The signature set
\[
\mathcal S_R=\{s(R,\theta):\theta\in\Theta_R\}
\]
contains the population signatures that regime $R$ can generate. For the true signature $s_0$,
\begin{equation}
\mathcal R(s_0)=\{R\in\mathcal R:s_0\in\mathcal S_R\}
\label{eq:signature-class-set}
\end{equation}
is the set of physical regimes compatible with those features. A singleton identifies the regime;
$\mathcal R(s_0)=\{I,S\}$ leaves inert and stabilized motion unresolved.

An astronomer estimates these features from a finite, noisy record. If the signature estimate
has error at most $\varepsilon$ in a normalized metric, a distance greater than $2\varepsilon$
from $s_0$ to every rival signature set guarantees unique compatibility. One radius covers
estimation error and the other the tolerance for matching a candidate regime. A more precise
estimate permits classification once the chosen features separate the physical explanations, while
coarse measurements can leave all three regimes compatible.

Whether restrictions on forcing and correction yield observable separation depends on how
those inputs propagate into the measured residual.

\subsection{Physical and observation model}

Radiation pressure, thermal recoil, or a body-generated pulse changes an object's velocity;
subsequent motion accumulates that change into a displacement. Corrective thrust can offset part
of the input and hence part of its propagated displacement. The observed position records the
net effect, rather than the two contributions separately. Let $\mathbf x_t$ be the object's
position and $\mathbf x_t^{\mathrm{grav}}$ its gravity-only reference position in the same
spatial coordinates. In a linearized propagation model, write the accumulated displacement due
to forcing as $\mathbf s_t$ and the corrective offset as $\mathbf o_t$
\citep{andree2026stability}:
\begin{equation}
\mathbf{e}_t\equiv\mathbf{x}_t-\mathbf{x}^{\mathrm{grav}}_t
=\mathbf{s}_t-\mathbf{o}_t.
\label{eq:resid}
\end{equation}
All terms in equation~\eqref{eq:resid} have units of length: $\mathbf{s}_t$ and
$\mathbf{o}_t$ are propagated displacements, not instantaneous forces or accelerations.
The additive decomposition is valid while both displacements remain within the linearization horizon of
Appendix~\ref{app:longhorizon}. The observation model is
$\mathbf y_t=P_t\mathbf e_t+\mathbf n_t^{\mathrm{obs}}$, where $P_t$ maps spatial displacement
to the measured coordinates and $\mathbf n_t^{\mathrm{obs}}$ is measurement noise. For optical
astrometry, $P_t$ includes the sky-plane projection and the conversion from displacement to angle.
Thus the candidate forcing, correction, and measurement laws jointly determine
$P^{(T)}_{R,\theta}$. Moments of the residual alone cannot recover its two physical contributions.

The corrective offset $\mathbf o_t$ is produced by actuator commands.
Figure~\ref{fig:authority} illustrates the command geometry motivated by
\citet{andree2026thruster}, distinguishing balanced thrust, disturbance rejection, and directional
steering. The reachable set specifies which commands are available; a stochastic correction policy
specifies how they are used.

A balanced command can conceal substantial actuator activity in the net motion. More generally,
correction can offset a disturbance so closely that the remaining displacement approaches zero.
Weak forcing without correction approaches the same limit. Before using quiet motion to identify
an inert object, these two explanations must therefore be compared at the level of the measured
record. Lemma~\ref{lem:zeroresidual} shows that both limits leave only the common measurement
noise, and that exact cancellation and unforced propagation give identical observation laws.

\noindent\begin{minipage}{\textwidth}
\captionsetup{type=figure}
\centering\singlespacing
\includegraphics[width=.925\linewidth,height=.55\textheight,keepaspectratio]{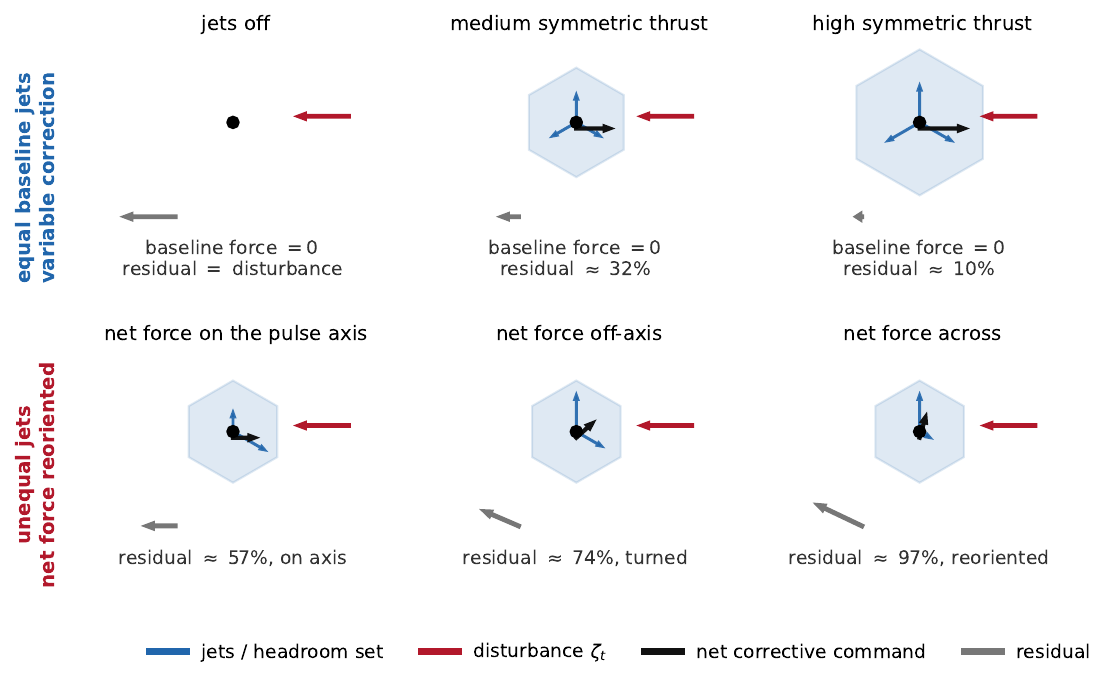}
\caption{\textbf{Actuator commands and reachable sets.} The commands generate the corrective offset in equation~\eqref{eq:resid} \citep{andree2026thruster}.
Equal baseline thrusts have zero net force. Modulation produces a net command; its mean depends on the
policy and disturbance law. The top row illustrates deviations around the baseline, while the bottom row
illustrates directional commands from individual thrust reductions. The independent-deviation and
constant-total-thrust constraints permit different sets of force commands, termed reachable sets
(Appendix~\ref{app:authority}).
Blue arrows denote thrusts, dark arrows corrective commands, pale arrows remaining disturbances, and red
arrows incoming disturbances. The jets-off panel has no actuator correction. Arrow lengths are schematic
and use a common illustrative scale.}
\label{fig:authority}
\end{minipage}
\par\medskip

\Needspace{5\baselineskip}
\begin{lemma}[Common zero-residual limit]
\label{lem:zeroresidual}
Fix a nonempty finite observation window $\mathcal I$ and square-integrable displacement records.
For such a record define
\[
\|\mathbf u\|_{\mathcal I}
=\left(\frac{1}{|\mathcal I|}\sum_{t\in\mathcal I}\E\|\mathbf u_t\|_2^2\right)^{1/2},
\]
with the expectation omitted for deterministic records. In equation~\eqref{eq:resid},
both of the following sequences converge to the same zero-residual limit:
\[
\begin{aligned}
\|\mathbf s^{(n)}-\mathbf o^{(n)}\|_{\mathcal I}\to0
&\quad\Longrightarrow\quad\|\mathbf e^{(n)}\|_{\mathcal I}\to0,\\
\mathbf o^{(n)}=0,\quad\|\mathbf s^{(n)}\|_{\mathcal I}\to0
&\quad\Longrightarrow\quad\|\mathbf e^{(n)}\|_{\mathcal I}\to0.
\end{aligned}
\]
If residual observations satisfy
$\mathbf y_t^{(n)}=P_t\mathbf e_t^{(n)}+\mathbf n^{\mathrm{obs}}_t$ with fixed bounded linear
observation maps $P_t$ and the same square-integrable measurement noise $\mathbf n^{\mathrm{obs}}_t$,
both sequences satisfy $\|\mathbf y^{(n)}-\mathbf n^{\mathrm{obs}}\|_{\mathcal I}\to0$.
Exact cancellation $\mathbf o=\mathbf s$ and unforced propagation
$\mathbf o=\mathbf s=0$ therefore have identical observation laws under this model.
\end{lemma}

The proof is given in Appendix~\ref{proof:zeroresidual}.

Thus increasing precision cannot distinguish the two exact endpoints under this observation
model. Away from the limit, the residual can contain information about how correction responds
to disturbances, even when its mean amplitude is small. Opposing physical inputs can also cancel
before propagation, so small displacement alone remains insufficient. The useful comparison is
how the remaining motion changes when the disturbance history changes, with the measurement law
held fixed.

The distinction between smooth forcing and pulses depends on the time scale resolved by the
observations. Radiation pressure and secular thermal recoil contribute to a smooth mean input
\citep{burns1979radiation,vokrouhlicky2015yarkovsky}. Outgassing can likewise have a smooth
orbit-scale mean, represented by the Marsden law \citep{marsden1973nongrav}, with discrete
activity superimposed on it. Write $\mathbf{s}_t=\fext_t+\fbody_t$ for the propagated
contributions of the smooth-mean and pulsed channels. The superscript $\mathrm{ext}$ labels the
first channel rather than an exclusively external physical origin; modeled gravitational
perturbations remain in $\mathbf{x}^{\mathrm{grav}}_t$. A law for the mean force leaves open
how its unpredictable component varies, which is the part needed for the variance and shape
comparisons.

At each sampling step, part of the velocity change can be predicted from the preceding record,
and part remains uncertain. Let $\Fcal_t$ denote the information available through step $t$.
An input innovation is the velocity increment minus its conditional mean given
$\Fcal_{t-1}$. This separation distinguishes changing the intended motion from changing its
unpredictable variation. A predictable additive increment $\mathbf b_t$, measurable with
respect to $\Fcal_{t-1}$, changes one update's conditional mean while leaving its innovation
covariance unchanged when that innovation law is held fixed. Returning a displaced body to an
earlier path would require a policy governing later mean inputs as well. The local impulse
comparison holds those later mean inputs fixed across the two continuations, so that the retained
impulse and the response of subsequent variability can be examined separately.

A controller can reduce unpredictable motion by opposing the disturbance detected within a
sampling step. Sensing and executing that correction can introduce a new error, however, so
stronger cancellation need not simply rescale the original disturbance. Let $\bm\zeta_t$ be
the uncorrected velocity innovation and $\bm\nu_t$ the sensing-and-actuation noise, both in
velocity units. In the linear reference law, cancellation of a fraction $\lambda$ leaves the
uncanceled disturbance together with the corresponding correction error:
\begin{equation}
\mathbf{r}_t=(1-\lambda)\bm{\zeta}_t+\lambda\bm{\nu}_t,
\label{eq:withinstep}
\end{equation}
where the disturbance and noise have conditionally zero means and are conditionally independent
given $\Fcal_{t-1}$; the noise is memoryless across steps. The share $\lambda\in[0,1]$ is
constant over the compared windows. This law changes the relative contributions of disturbance
and correction noise, which can change distributional shape as well as variance.

The controller can also strengthen disturbance rejection after an event. Represent this second
action by a predictable recovery gain $K_t$ applied to the already-sensed remaining input
$\mathbf r_t$. Unlike the within-step mixture, this stage introduces no additional noise in
the reference law: it rescales both existing contributions. The final velocity innovation and
its conditional covariance are
\begin{equation}
\bm{\varepsilon}_t=(1-K_t)\mathbf{r}_t,\qquad
H_t=(1-K_t)^2\Sigma_t,\qquad
\Sigma_t=\operatorname{Var}(\mathbf{r}_t\mid\Fcal_{t-1}),\quad 0\le K_t\le K_{\max}<1.
\label{eq:authority}
\end{equation}
The covariance $\Sigma_t$ is measured before recovery scaling and generally differs from the
inert velocity-innovation covariance $g(W_t)$, where $W_t$ indexes the forcing environment.
Its quiet value is denoted by $\Sigma$. Raising the gain can reduce variance below this
same-history quiet baseline. Although $K_t$ is predictable, executing $K_t\mathbf r_t$ requires
access to the current remaining input. Additional execution noise changes the scaling law and
its conditional moments.

Equations~\eqref{eq:withinstep}--\eqref{eq:authority} are exact for the stated linear stochastic
model. A finite actuator clips commands outside its reachable set, and the clipped model generally
has different conditional moments. The sensing floor is the correction system's noise contribution.
A nondegenerate Gaussian law for this noise has unbounded support, so finite available corrective
thrust cannot implement these linear laws on every realization.
Appendix~\ref{app:authority} gives the realized-command constraint and separates this implementation
limit from the idealized moment identities. Zero conditional mean of the linear correction follows
from the specified noise laws; symmetry of an actuator's feasible set alone does not impose it.

\subsection{Disturbance dynamics and trajectory propagation}

An uncorrected body can experience a period of larger disturbances because its environment has
changed. A body-generated episode can also remain active after an initial pulse, so that recent
activity contains information about the size and direction of subsequent fluctuations. These
possibilities differ from a controller that suppresses later disturbances. To separate them, the
input model must distinguish variability associated with the environment from variability
associated with the object's own recent shocks.

At the sampling interval $\Delta t$, write the velocity-input innovation as
$\bm\varepsilon_t=\bm\eta_t=H_t^{1/2}\bm\xi_t$. Here $H_t$ is its covariance conditional on
the past, predictable and positive definite on the standardized subspace, while
$\E[\bm\xi_t\mid\Fcal_{t-1}]=0$ and
$\E[\bm\xi_t\bm\xi_t'\mid\Fcal_{t-1}]=I$. The matrix $H_t$ therefore describes the size
and directional dependence of the next unpredictable velocity change. A $2+\nu$ moment for
some $\nu>0$ is a separate regularity condition; the shape comparisons additionally require
fourth moments.

A BEKK-type recursion \citep{engle1995bekk} represents three sources of changing uncertainty:
the latest shock can alter subsequent activity, elevated variability can persist, and the
external environment can change the forcing even without either form of feedback. Write
\begin{equation}
H_t \;=\; C \;+\; \Acal\,\bm{\eta}_{t-1}\bm{\eta}_{t-1}'\,\Acal' \;+\; \Bcal\,H_{t-1}\,\Bcal' \;+\; G(W_t),
\label{eq:hrec}
\end{equation}
with $C\succ0$ and $G(W_t)\succeq0$. The baseline $C$ and environmental contribution $G(W_t)$
specify variability present without feedback from the object's own shocks. The term containing
$\Acal$ allows a recent velocity innovation to raise subsequent covariance, with the matrix
mapping its direction into later variability. The term containing $\Bcal$ carries preceding
covariance into the next step. Their quadratic form preserves a valid covariance matrix. For an
impulsive body channel, these terms can represent continuing or self-exciting activity;
asymmetric news-impact models allow other responses to recent shocks \citep{engle1993news}.
The chosen lags and quadratic response are statistical representations of such behavior,
rather than physical laws that every outgassing body must satisfy.

The inert comparison removes both forms of internal covariance feedback by setting
$\Acal=\Bcal=0$, leaving $H_t=g(W_t)=C+G(W_t)$. The environment can still be persistent or
time-varying. The restriction says that, once it is held fixed, a realized kick changes the
body's motion without itself changing the covariance of later forcing. That distinction permits
an impulse to displace an inert body while leaving its subsequent innovation variance unchanged.
Other positive covariance models depending only on the matched environment permit the same
comparison under the stated moment conditions. Finite second moments and strict stationarity
require separate hypotheses \citep{bougerol1992strict,engle1995bekk}. Stabilization acts on the
remaining disturbances through equation~\eqref{eq:authority}; its suppression law is therefore
specified separately from the BEKK recursion.

A velocity change and its positional effect have different persistence. Once a kick has changed
velocity, the displacement continues to accumulate even if no further kick occurs. On a short
horizon where differential gravitational curvature can be neglected, this is represented by
stacking the displacement and residual velocity as
$\mathbf z_t=(\mathbf e_t',\dot{\mathbf e}_t')'\in\R^6$. With the velocity kick applied at
the end of a sampling step,
$\mathbf z_{t+1}=\mathcal T\mathbf z_t+
(\mathbf0',(\mathbf b_{t+1}+\bm\varepsilon_{t+1})')'$, where
\begin{equation}
\mathcal T=\begin{pmatrix}I_3&\Delta t\,I_3\\0&I_3\end{pmatrix},
\qquad
\mathcal T^h=\begin{pmatrix}I_3&h\Delta t\,I_3\\0&I_3\end{pmatrix}.
\label{eq:jordan}
\end{equation}
The off-diagonal block accumulates velocity into position. Thus a retained impulse produces a
local displacement ramp, while any change in subsequent innovation variance concerns new inputs,
not removal of the initial offset. The predictable mean input $\mathbf b_{t+1}$ allows a
separately specified mean policy. The gravity gradient limits the horizon over which this
straight-line relative-motion approximation is valid.

Delayed forcing or correction can extend over several sampling steps. Finite-lag input
recursions represent that persistence, and error-correction models can represent adjustment among
velocity components. Under suitable root, driver, and moment conditions, a stable transition
matrix $\Mcal$ yields a stationary input process that forgets its initialization
(Appendix~\ref{app:framework}). Yet position still integrates those inputs. A body can therefore
have regular velocity innovations while its displacement from the original reference continues
to grow. Predicting the next measured position requires an observer that tracks velocity as well
as position; stationarity of its prediction error is a separate property of that object--observer
pair. Its covariance is obtained by propagating $H_t$ through the state dynamics, measurement
map, and observer, including measurement noise, rather than by equating it with $H_t$ or $g(W_t)$.

Cadence determines which of these inputs are resolved as individual innovations. At a coarser
cadence, a moving-average input may take the form
$\mathbf w_t=\sum_{j=0}^{q_{\mathrm{MA}}}\Theta_{j,t}\bm\eta_{t-j}$, with $\Theta_{0,t}=I$.
For positive lag order, part of this input is already represented in past information, so
$\mathbf w_t$ is not generally an innovation relative to $\Fcal_{t-1}$. The propositions use
$q_{\mathrm{MA}}=0$ at their stated cadence, keeping the covariance comparisons attached to the
unpredictable input at that scale.

\subsection{Stochastic assumptions and reference regimes}
\label{sec:regimes}

The physical decomposition alone permits all three regimes to produce similar displacement.
Separation requires restrictions on how their inputs vary and on what the observer retains of
that variation. Inert propagation has no corrective offset or pulsed body input; impulsive forcing
adds the latter; stabilization acts on the disturbances. Effects that are smooth at the diagnostic
cadence remain in the smooth-mean channel, irrespective of their physical origin. The assumptions
below specify the variability needed to distinguish these resolved mechanisms.

An accumulated displacement need not make the next observation increasingly difficult to
predict: a velocity-tracking observer can account for the retained drift. The relevant error is
therefore the departure from the next-position forecast, not the entire distance from the original
gravity-only path. Let $\mathbf q_t$ denote the centered one-step prediction error
$\mathbf y_t-\hat{\mathbf y}_{t\mid t-1}$, or the centered output of a specified transformation
of that error, for a fixed observer. Comparing its statistical behavior across a diagnostic
window requires a stable error law after the observer's initialization transient. Write
$\Gamma(0)=\E[\mathbf q_t\mathbf q_t']$ for its contemporaneous covariance.
Assumption~\ref{ass:A1} imposes the stationary law for the inert reference while retaining an
unpredictable component.

\Needspace{5\baselineskip}
\begin{assumption}[Stationary forecast residual in the inert limit]
\label{ass:A1}
After the fixed velocity-tracking filter has converged, the centered forecast-residual process
used for the regularity diagnostic is indexed on $\mathbb Z$, covariance stationary, and has covariance
$\Gamma(0)\succ0$. It is purely non-deterministic: the intersection of its closed linear past
spaces is $\{0\}$. If the raw residual is heteroskedastic, the diagnostic is applied to a stated transformed
process, and these properties are assumed for that process itself. Standardization by an input
covariance $g(W_t)$ does not by itself establish them for the filter output.
The transient before convergence is excluded from the diagnostic window.
\end{assumption}

A converged observer provides a consistent way to remove predictable motion. Riccati convergence
under appropriate detectability, stabilizability, and stationary-noise hypotheses motivates such
a filter \citep{anderson1979optimal,harvey1989forecasting}, while the stated error properties
remain assumptions about the joint forcing and observation model. Pure non-determinism rules out
a nonzero component that is perfectly reconstructible from arbitrarily remote observations.
This permits a strict bound on linear predictability without requiring a particular distribution.
It also leaves open the physical source of the regular errors: weak uncorrected forcing and
corrected forcing can both satisfy these conditions.

An isolated impulse provides a useful distinction only if its kinematic effect can be separated
from changes in the forcing that follows it. For the inert reference, the impulse changes
velocity, but later uncertainty is still governed by the same environment. This requires the
no-own-shock-feedback restriction described above. Radiation pressure's mean $r^{-2}$ dependence
alone supplies no law for its stochastic variance. Assumption~\ref{ass:A2} therefore specifies
that variance separately and fixes the environmental comparison needed to attribute a later
variance response to something other than changing background forcing.

\Needspace{5\baselineskip}
\begin{assumption}[Inert innovation: exogenous, no own-shock feedback]
\label{ass:A2}
The centered inert velocity innovation has covariance
$H_t=g(W_t)=C+G(W_t)$, with $C\succ0$ and $G(W_t)\succeq0$.
It has no own-shock covariance feedback after conditioning on the environment: in
equation~\eqref{eq:hrec}, $\Acal=\Bcal=0$.
The index $W_t$ is $\Fcal_{t-1}$-measurable; predictability alone does not imply causal
exogeneity. An impulse-response comparison holds the environment fixed.
Testing this restriction requires an identified covariance model and a justified null calibration
appropriate to its boundary structure \citep{andrews2001boundary}.
\end{assumption}

Variance describes the typical squared size of an unpredictable kick; the fourth moment is more
sensitive to occasional large kicks. An intermittent input can therefore have high excess
kurtosis even when each pulse has bounded amplitude. To use that shape as evidence of an
additional body channel, however, both the frequency and amplitude distribution of its pulses
must be specified. A Bernoulli gate represents whether a pulse occurs, with a separate variable
for its amplitude. Remark~\ref{rem:kurtosis} gives the resulting shape measure and explains
why arrival frequency alone cannot order two channels.

\Needspace{5\baselineskip}
\begin{remark}[Sparsity and excess kurtosis]
\label{rem:kurtosis}
A Bernoulli gate gives a tractable model of intermittent activity, separating the frequency of
pulses from their amplitudes. Let $X=BJ$, where $B\sim\operatorname{Bernoulli}(p)$ is independent of a centered jump $J$,
$0<p\le1$, and $0<\E J^2<\infty$, $\E J^4<\infty$. Then
\[
\kappa_X=\frac{p\E J^4}{(p\E J^2)^2}-3=\frac{m}{p}-3,
\qquad m=\frac{\E J^4}{(\E J^2)^2}\ge1.
\]
Reducing $p$ increases excess kurtosis when $m$ is fixed. Comparing two different channels
requires comparing $m/p$, not arrival rates alone. Bounded jumps can produce arbitrarily large
positive excess kurtosis as $p\downarrow0$, so leptokurtosis does not establish heavy tails.
\end{remark}

For a fixed amplitude law, rarer pulses make the fourth moment large relative to squared
variance. Comparing body activity with external fluctuations additionally requires their amplitude
laws to be comparable in this standardized sense. This is a premise about unpredictable inputs,
not a consequence of smooth mean forcing. A few dominant sample pulses do not establish finite
population fourth moments or a tail class. The shape comparison is made only in laws with
positive variance and finite fourth moment, and uses the following cross-channel restriction.

\noindent\emph{Standing condition (cross-regime kurtosis ordering).}
When comparing (C) with (I), the external component of (C) has the same standardized law as the
inert channel in the same conditioning set. Assume $\kappa_I\ge0$ and $\kappa_b\le\kappa_I$
for the independent body channel. For Bernoulli-thinned channels the latter condition is
$m_b/p_b\le m_I/p_I$. These are additional restrictions; $p_I<p_b$ alone is insufficient.
This ordering selects a body channel no more leptokurtic than the matched external channel.
Independence makes fourth cumulants additive, so the standardized shape change can be traced to
variance shares. The comparison $\kappa_C\le\kappa_I$ applies to any component distributions
satisfying these conditions, rather than requiring Bernoulli-gated pulses.

For the variance comparison, elevated-forcing windows are selected so that
$\E[\operatorname{tr}g(W_t)\mid\mathrm{event}]>
\E[\operatorname{tr}g(W_t)\mid\mathrm{quiet}]$ whenever a positive inert event increment is
required. This window-level condition is additional to the exogenous covariance specification.

For a corrected body, the initial displacement and the response to later disturbances need not
recover together. Part of an incident impulse can remain in the velocity even while the controller
strengthens its rejection of new fluctuations. To examine these effects separately, the reference
comparison fixes the retained initial impulse, allows the recovery gain to rise, and holds the
subsequent mean inputs equal across the two continuations. Within-step cancellation adds a second
comparison between the disturbance it removes and the sensing noise it introduces.
Assumption~\ref{ass:A3} specifies these actions, fixes the noise law over the compared windows,
and separately maintains the observer's forecast-error regularity.

\Needspace{5\baselineskip}
\begin{assumption}[Linear stochastic correction law]
\label{ass:A3}
In the stabilized reference model:
\begin{enumerate}[label=\textbf{(\alph*)},leftmargin=*,itemsep=0.25em]
\item \emph{Recovery gain.} Within an isolated response episode, with $h$ counting sampling steps
since the impulse, $K_t=K(h)$ is predictable,
$K(0)=0$, nondecreasing and bounded by $K_{\max}\in(0,1)$. Its idealized continuation tends to
$K_{\max}$. The quiet comparison has $K=0$. For the explicitly identified constant-gain
variant, the same fixed $K\in[0,1)$ applies in both quiet and event windows instead.
\item \emph{Impulse and mean correction.} Predictable additive corrections change one-update
conditional means at a fixed innovation law. A separate test intervention supplies an exogenous
velocity impulse whose onset-retained value is $\eta\bm\delta$, with fixed $\eta\in(0,1]$.
This collinear retention is an additional premise, not a consequence of the stochastic disturbance law.
No subsequent impulse-induced mean steering acts during the local response horizon.
The fully canceled case $\eta=0$ has no ramp and is excluded from that calculation.
\item \emph{Filter regularity.} The centered forecast-residual process satisfies all the covariance,
stationarity and pure-non-determinism conditions of Assumption~\ref{ass:A1} under (S) as well.
This is maintained, not deduced from the gain schedule or actuator symmetry.
\item \emph{Moment law.} Equations~\eqref{eq:withinstep}--\eqref{eq:authority} hold as linear
stochastic laws. Given $\Fcal_{t-1}$, $\bm\zeta_t$ and $\bm\nu_t$ are centered and independent,
with finite covariance. The share $\lambda$ is constant over compared windows; the noise
covariance $F_\nu$ is constant there. The quiet covariance before recovery scaling is $\Sigma\succ0$.
Where fourth moments are used they are finite with positive projected variances. A Gaussian
floor or positive disturbance excess kurtosis is imposed explicitly by the particular result,
not inferred from the term ``near-Gaussian.'' For an event-level variance reduction with
$\lambda>0$, additionally require
$\lambda\operatorname{tr}F_\nu<(2-\lambda)
\operatorname{tr}\operatorname{Var}(\bm\zeta_t\mid\Fcal_{t-1})$.
\end{enumerate}
Finite hard saturation is a separate implementation, discussed in Appendix~\ref{app:authority};
its moments must be recomputed. Exact Gaussian identities do not assert globally feasible
linear cancellation by a finite actuator.
\end{assumption}

These conditions allow a displaced mean to continue along its local ramp while the unpredictable
part of later velocity increments becomes smaller. The effect of the retained fraction $\eta$
is therefore different from that of the recovery gain $K(h)$. The within-step share $\lambda$
also changes the mixture of physical disturbance and correction noise. Clause~(d)'s optional
trace bound requires the removed disturbance variance to outweigh the added sensing variance;
without it, correction can attenuate an event increment without lowering every variance level.
Coordinatewise attenuation requires the corresponding directional inequalities. The episode-level
response remains distinct from the across-episode forecast-error regularity in clause~(c),
which depends on the joint inputs, gain schedule, and initialization.

Physical feasibility also depends on mass, command duration, and direction
(Appendix~\ref{app:authority}); the gain law alone does not determine the retained fraction
$\eta$.

Definition~\ref{def:regimes} combines these stochastic laws with the presence or absence of body
pulses and correction. It classifies the resolved dynamics rather than assigning an origin from
appearance.

\Needspace{5\baselineskip}
\begin{definition}[Reference residual structures]
\label{def:regimes}
\textbf{(I)} \emph{Inert limit}: $\mathbf o_t=0$ and $\fbody_t=0$, so
$\mathbf e_t=\fext_t$ is the displacement contribution of the smooth-mean input channel, with
no body-generated pulse or corrective offset. Its centered velocity innovations follow
Assumption~\ref{ass:A2}. The split into smooth and pulsed channels is defined at the declared
sampling cadence, so the inert label is a cadence-relative operational class: a slow modulation
resolved only at a finer cadence would enter the impulsive channel there.
\textbf{(C)} \emph{Impulsively forced limit}: $\mathbf o_t=0$ and
$\mathbf e_t=\fext_t+\fbody_t$, with nonzero pulsed body inputs. At the innovation level,
$\bm\varepsilon_t=\bm\zeta^e_t+\bm\zeta^b_t$. Conditional independence, finite moments,
and positive event amplification are imposed where a result uses them.
\textbf{(S)} \emph{Stabilized reference limit}: $\mathbf o_t\not\equiv0$ and
$\mathbf e_t=\mathbf s_t-\mathbf o_t$, with the input law of Assumption~\ref{ass:A3}.
Forecast-residual regularity is a separate premise; neither stationarity of $\mathbf e_t$ nor
long-horizon return follows from a stable input-adjustment block alone.
\end{definition}

Intermediate cases combine the components of the residual equation. Body-generated effects
that are smooth at the diagnostic scale are grouped with the inert model's smooth-mean input;
bursty effects enter the impulsively forced channel. Jets resemble the stabilized reference
model when they correct disturbances. Their geometry alone supplies no such response law.

\subsection{Diagnostic targets for regime separation}
\label{sec:diagnostics}

An inert body follows the motion produced by its forcing; an impulsively forced body adds
pulses of its own; a stabilizer offsets part of the disturbance. The four diagnostics, O1--O4,
ask how these differences appear in a body's ability to maintain a predictable course and
respond to changing conditions.

\emph{O1, regular prediction errors.} A spacecraft assigned to follow a course must keep its
departures within an acceptable range. It estimates its state and adjusts its motion, making
prediction part of the control task. O1 concerns the errors in the analyst's forecast of the
next position: these can remain regular even after the body has shifted from its earlier
course. Proposition~\ref{prop:o1} gives the stable lag-one projection for the stationary
reference. Uncorrected motion can also produce regular prediction errors, so O1 provides a
baseline shared by the candidate mechanisms.

For the control dynamics, requiring an amplification factor below one at every update can be
relaxed to contraction on average. Errors may grow over some intervals, with factors above one,
provided contraction dominates over time under the Bougerol-type expected log-contraction conditions in
Appendix~\ref{app:stability} and Remark~\ref{rem:frequency}. During those intervals, the
spacecraft must still stay close enough to its intended course. Remark~\ref{rem:o1phys} expresses this separate requirement as a
bound on departures over the planning horizon.

\emph{O2 and O3, responding to stronger disturbances.} When forcing intensifies, an uncorrected
body receives the additional disturbance without acting against it. A stabilizer can oppose
part of it, as a pilot or control system responds when an aircraft enters stronger turbulence.
O2 asks how much of the increased variability remains in the motion. The comparison holds the
forcing exposure fixed: the reference is the same body under the same disturbance with
correction switched off. It is this comparison, rather than quiet motion alone, that identifies
attenuation under the conditions of Proposition~\ref{prop:o2o3}.

O3 asks whether unusually large increments become less prominent relative to the remaining
fluctuations. This complements the variance comparison: shrinking every increment by the same
fraction lowers variance but leaves its relative shape unchanged. Mixing the remaining
disturbance with sensing noise can also reduce the prominence of large increments, as the
kurtosis-share identity and conditional-shape result establish
(Lemma~\ref{lem:kurtshare}; Proposition~\ref{prop:o3inv}). Added body pulses can change shape
too. Variance and shape therefore have to be read together against the specified uncorrected
benchmark: suppression may appear as a smaller variance increment, a change in relative shape,
or both.

\emph{O4, the course after an impulse.} A single retained kick changes an inert body's
velocity, so its displacement from the earlier course initially grows along a smooth ramp.
Gravity subsequently curves the path and can bring it closer to the reference in some orbital
configurations (Appendix~\ref{app:longhorizon}). Further body-generated pulses can prolong or
redirect the response, although weak or infrequent activity can still look smooth at the
observed cadence (Remark~\ref{rem:cometary}). Smoothness alone therefore leaves several
physical explanations compatible.

A course-tracking controller has two tasks after a large displacement: establish the course to
follow and limit the effect of new disturbances along it. O4 separates these tasks. In the
rising-gain reference, the retained impulse leaves a smooth local displacement (O4a), while
stronger correction reduces subsequent velocity fluctuations (O4b). Proposition~\ref{prop:o4}
formalizes this combination. Returning to the old course requires additional steering; the
controller could instead adopt the shifted course (Remark~\ref{rem:adoption}). A controller
already operating at a steady gain can show no further variance reduction after the impulse.
Its correction is assessed through the matched comparisons of O2--O3 or a variance-level
comparison (Remark~\ref{rem:smodes}).

Each candidate must account for both the observed course and the fluctuations under the same
forcing and observing conditions. Table~\ref{tab:diagnostics} connects these physical
comparisons to their statistical roles.

The course tolerance and the statistical regularity concern different errors: one measures
departure from the intended path, the other error in predicting the next observation.
Remark~\ref{rem:o1phys} states the operational criterion explicitly.

\Needspace{5\baselineskip}
\begin{remark}[Statistical stationarity and physical tolerance]
	\label{rem:o1phys}
	Fix an observed coordinate or proxy $u_t$, an observation window $\mathcal W$, a prescribed
	physical tolerance $L>0$, and, for a probabilistic criterion, a tail level $\alpha\in(0,1)$.
	A realized path meets the tolerance when $\max_{t\in\mathcal W}\|u_t\|\le L$; a model meets
	the corresponding probabilistic requirement when
	$\Pr(\max_{t\in\mathcal W}\|u_t\|>L)\le\alpha$.
	These are finite-window physical criteria. Covariance stationarity instead means finite,
	time-invariant means and lag covariances; it implies neither criterion for a given $L$.
	Conversely, satisfying a tolerance on one observed window does not establish stationarity.
\end{remark}

\par\medskip
\noindent\begin{minipage}{\textwidth}
\captionsetup{type=table}
\centering\small\singlespacing
\setlength{\abovecaptionskip}{4pt}
\caption{\textbf{The four diagnostics and their roles.} When the independent-channel
comparison gives a strict kurtosis gap, a match to the inert shape benchmark or a strict
twin-suppression contrast excludes impulsive forcing. Suppression contrasts distinguish
stabilized from inert motion; forecast-residual regularity is maintained.}
\label{tab:diagnostics}
\begin{tabularx}{\textwidth}{@{}l >{\raggedright\arraybackslash}p{4.6cm} >{\raggedright\arraybackslash}X@{}}
\toprule
Diagnostic & What it compares & Role in identification \\
\midrule
O1 forecast-residual regularity & regularity of the one-step prediction errors & Shared statistical premise; can also hold under impulsive forcing \\
O2 disturbance attenuation & event-versus-quiet variance against an uncorrected twin & Strict attenuation excludes both uncorrected regimes under the comparison law \\
O3 excess-kurtosis reduction & scalar standardized fourth cumulant of the errors & The cross-regime shape gap separates impulsive forcing from inert motion; twin-shape suppression detects correction \\
O4 impulse response & level (O4a) and variance (O4b) response to an isolated disturbance & The shared local ramp and the rising-gain variance response jointly distinguish the stated inert and stabilized models \\
\bottomrule
\end{tabularx}
\end{minipage}
\par\medskip

\section{Results}
\label{sec:results}

Regular prediction errors leave all three mechanisms possible. The required distinction is
how each mechanism changes the effect of the same disturbances, with the observing process
held fixed.

\subsection{Forecast-residual stationarity}
\label{sec:o1}

A retained velocity offset generates a predictable displacement ramp. Once that velocity is
tracked, the remaining one-step errors can be regular under either inert or corrected motion.
O1 uses a fixed observer containing position and velocity but no body-acceleration state; its
noise specification, initialization, and diagnostic transformation are also fixed. The question
is what the assumed regularity implies for predictability, before it is interpreted as evidence
about a physical mechanism.

For the centered stationary residual $\mathbf q_t$, the population lag-one projection
$\Phi=\Gamma(1)\Gamma(0)^{-1}$, with
$\Gamma(j)=\E[\mathbf q_t\mathbf q_{t-j}']$, gives the best linear prediction from the preceding
residual. Stationarity bounds the variance that this prediction can contain. Pure non-determinism
additionally excludes a nonzero component predicted without error from arbitrarily remote
observations. Proposition~\ref{prop:o1} uses these two properties to obtain a strict spectral
bound shared by the inert and stabilized references.

\Needspace{5\baselineskip}
\begin{proposition}[Lag-one projection of the inert and stabilized limits]
\label{prop:o1}
For the centered forecast-residual process specified in Assumption~\ref{ass:A1}, and in
Assumption~\ref{ass:A3}(c) for (S), the population lag-one operator satisfies $\rho(\Phi)<1$.
This is a shared property of (I) and (S). The proposition imposes no separation condition
on the impulsively forced class (C).
Neither a deterministic physical-tolerance bound nor a particular finite-sample test outcome
follows from this proposition.
\end{proposition}

The proof is given in Appendix~\ref{proof:o1}.

The bound restricts linear predictability rather than describing a physical return force.
The projection error can retain serial dependence, and stationary long-memory responses can
decay more slowly than $\Phi^h$. These possibilities remain compatible with the strict bound.

Empirical tests of regularity address narrower hypotheses. A componentwise augmented
Dickey--Fuller test addresses a scalar root at $+1$ \citep{dickey1979unitroot}; a stationarity
test uses a stationary null \citep{kwiatkowski1992kpss}. Neither directly tests the entire
multivariate boundary $\rho(\Phi)=1$, which also includes $-1$ and complex unit-modulus
eigenvalues. Their outcomes require calibration for the specified observer and error law.
A nonstationary fractional model with $d\in(1/2,1)$ falls outside the proposition, motivating
memory estimation across stationary and nonstationary regions \citep{shimotsu2005whittle}.
Filter tuning, cadence, and memory specification therefore affect whether the observed record
supports the maintained regularity premise.

An uncorrected body can be easy to predict while drifting from its earlier course.
Proposition~\ref{prop:o1} therefore establishes a shared prediction property, not evidence of
active course-holding. A controller can also tolerate temporary error growth when later
contraction compensates for it; Appendix~\ref{app:stability} gives the recurrence conditions.
The intervening departures must still meet the planner's tolerance
(Remark~\ref{rem:o1phys}). Predicting the motion and keeping it near a chosen course answer
different questions. To identify correction, the next comparison asks whether less disturbance
reaches the motion than the same forcing would produce without control.

\subsection{Stabilization and disturbance variance}
\label{sec:o2o3}

Low variability alone leaves weak forcing and strong correction unresolved. The distinction
becomes testable when the forcing exposure is fixed: how much would the same object fluctuate
if its disturbances were propagated without correction? Define its \emph{unstabilized twin}
by setting $\mathbf{o}_t\equiv\mathbf{0}$ while retaining the same forcing. For a body with
pulsed inputs, those inputs remain in the twin. This comparison attributes attenuation to
correction rather than to a different disturbance law. An empirical approximation constructed
from other bodies or simulation additionally needs matched physical coupling and observation
conditions.

A period of stronger forcing provides an increment against which correction can be measured.
Write
$\Delta H=\E[\operatorname{tr}H_t\mid\mathrm{event}]
-\E[\operatorname{tr}H_t\mid\mathrm{quiet}]$.
This averages the uncertainty remaining at each history before comparing event and quiet
windows; pooling all observations can also mix different conditional means and scales.
The shape comparison must likewise specify a conditional law, an average of conditional
kurtoses, or a pooled-window law. Superscript $\mathrm{pass}$ denotes the same forcing
propagated without correction. Since each uncorrected regime is its own twin, its twin contrast
is zero; the question is whether correction produces a strict departure from that equality.

The shape effect is less immediate than the variance effect. Cancellation reduces the disturbance,
but the correction system introduces its own noise; a body-generated pulse instead adds another
forcing component. In both cases, the observed increment combines channels, so the relevant
question is how their relative contributions determine standardized shape. For independent
channels, fourth cumulants add while standardization divides by the square of total variance.
Lemma~\ref{lem:kurtshare} expresses the result through variance shares, allowing the same
calculation to compare added body activity and disturbance cancellation.

\Needspace{5\baselineskip}
\begin{lemma}[Kurtosis share identity]
\label{lem:kurtshare}
Fix a coordinate or nonzero deterministic linear projection of the two vector channels, and write
the resulting centered scalar variables as $Z,N$. In the probability law under consideration
(unconditional, or conditional on a specified history), assume independence, positive variances
$\sigma_Z^2,\sigma_N^2$ and finite fourth moments. For a fixed $\lambda\in[0,1]$, let
$R=(1-\lambda)Z+\lambda N$. Then
\[
\kappa_R=w^2\kappa_Z+(1-w)^2\kappa_N,\qquad
w=\frac{(1-\lambda)^2\sigma_Z^2}{(1-\lambda)^2\sigma_Z^2+\lambda^2\sigma_N^2}.
\]
This is a scalar identity. It does not give the kurtosis of a vector norm, nor does conditional
independence imply independence in a pooled window.
\end{lemma}

The proof is given in Appendix~\ref{proof:kurtshare}.

An independent zero-excess noise floor contributes variance but no fourth cumulant, reducing the positive
excess kurtosis of the remaining disturbance. For two pulsed channels, both cumulants contribute.
Applying Lemma~\ref{lem:kurtshare} to these alternatives and comparing event with quiet variance,
Proposition~\ref{prop:o2o3} establishes when correction suppresses the twin response and which
component restrictions distinguish impulsive from inert shape.

\Needspace{5\baselineskip}
\begin{proposition}[Disturbance-conditioned variance and kurtosis comparisons]
\label{prop:o2o3}
Compare fixed quiet and event sets with finite moments. Let $V_t$ be the trace of the
unstabilized twin's conditional covariance and suppose
$\Delta V:=\E[V_t\mid\mathrm{event}]-\E[V_t\mid\mathrm{quiet}]>0$.
In (C), assume the external and body channels are conditionally independent and that their
trace-variance increments sum to a positive value. Then:

\emph{(i)} Each uncorrected regime is its own twin, so its variance and defined kurtosis increments
equal the corresponding twin increments; in particular $\Delta H=\Delta H^{\mathrm{pass}}>0$.

\emph{(ii)} In any conditioning set where the independent scalar-channel identity of
Lemma~\ref{lem:kurtshare} applies, (C) has
$\kappa_C=w_b^2\kappa_b+w_e^2\kappa_e$.
When the component channels are independent within each event and quiet window law,
the event-minus-quiet excess-kurtosis increment
is positive exactly when this weighted sum is larger for the event law. For averages of conditional
kurtoses, average the entire weighted expression in each window instead. Event amplification alone
does not establish either inequality.
Under the separate cross-regime ordering condition following Remark~\ref{rem:kurtosis},
$\kappa_C\le\kappa_I$.

\emph{(iii)} In the linear (S) model under Assumption~\ref{ass:A3}(d), with constant $\lambda$,
constant floor covariance, conditional independence, and $K=0$ in quiet windows,
$\Delta H\le(1-\lambda)^2\Delta V<\Delta V$ if $\lambda>0$.
For $\lambda=0$, $\Delta H<\Delta V$ holds exactly when
$\E[(1-(1-K_t)^2)V_t\mid\mathrm{event}]>0$.
With a constant gain in both windows, instead
$\Delta H=(1-K)^2(1-\lambda)^2\Delta V$, strictly below $\Delta V$ when $K>0$ or $\lambda>0$.

For shape, fix a conditional history or a window law satisfying the lemma's independence
conditions. If $\kappa_Z>0$, $\kappa_N=0$, both component variances are positive and
$0<\lambda<1$, then $0<\kappa_R=w^2\kappa_Z<\kappa_Z$.
Predictable recovery scaling preserves this \emph{conditional} ordering; pooled window kurtosis
requires a separate analysis of the variation in scales across histories. At $\lambda=0$
the twin shape is unchanged; at $\lambda=1$ the floor shape is attained.
\end{proposition}

The proof is given in Appendix~\ref{proof:o2o3}.

The variance result isolates the response to stronger forcing. Before recovery scaling, the
common sensing-noise variance cancels from the event-minus-quiet difference, whereas the physical
disturbance increment remains multiplied by $(1-\lambda)^2$. With zero quiet gain, recovery
removes an additional nonnegative amount of event variance; a common constant gain instead
multiplies the entire increment by $(1-K)^2$. Thus even a noisy correction system can respond
less strongly to an increase in forcing than its uncorrected twin. A lower absolute variance
level additionally requires that the removed disturbance variance outweigh the injected noise.
These comparisons keep the forcing law fixed; another body's lower variance could instead
reflect weaker physical coupling.

The shape result also separates physical amplitude from standardized variability. A second
independent pulsed channel can lower excess kurtosis because the increase in total variance
changes the normalization of both fourth cumulants. The ordering $\kappa_C\le\kappa_I$
requires the stated comparison of component shapes, rather than following from greater activity
alone. Correction with a zero-excess floor also lowers positive excess kurtosis, but does so
together with a twin-suppression contrast. Shape therefore needs to be read jointly with the
variance comparison when identifying the physical class. A nonzero-excess floor contributes
its own fourth cumulant, whose variance share determines the resulting comparison; none of
these scalar moment statements orders entire densities or asymptotic tail classes.

A remaining ambiguity concerns which part of the controller produces the shape change. Applying
a fixed fraction to every remaining disturbance shrinks small and large increments together,
so it need not alter standardized shape. Yet observations pooled across a recovery episode
combine periods of different variability, which can change sample shape even when each
conditional shape is preserved. Proposition~\ref{prop:o3inv} separates this pooling effect
from the mixture change caused by within-step cancellation. This distinction is needed to
interpret an observed shape reduction as evidence about a particular correction law.

\Needspace{5\baselineskip}
\begin{proposition}[Conditional shape invariance and its pooling limit]
\label{prop:o3inv}
Under Assumption~\ref{ass:A3}(d), with $c_t=1-K_t>0$ predictable:
\emph{(i)} every defined conditional standardized central moment of $c_t\mathbf r_t$
equals that of $\mathbf r_t$ in a fixed scalar projection. A predictable additive term changes
the conditional mean only. This does not assert invariance or monotonicity of pooled moments.
For the separate scale-mixture model $X=SZ$, assume $S\ge0$ is independent of centered $Z$,
$0<\E S^2<\infty$, $\E S^4<\infty$, $\E Z^2=1$ and $\E Z^4<\infty$. Then
\[
\kappa_X+3=(\kappa_Z+3)\frac{\E S^4}{(\E S^2)^2}\ge\kappa_Z+3.
\]
This compares a mixture with a constant-scale copy of $Z$. Comparing two mixtures requires
comparing their respective scale ratios.

\emph{(ii)} In a fixed conditional law, let the independent disturbance and floor have positive
finite variances, finite fourth moments, $\kappa_Z>0$ and $\kappa_N=0$.
As $\lambda$ increases from $0$ to $1$ with those laws fixed, the excess kurtosis of
$(1-\lambda)Z+\lambda N$ decreases strictly from $\kappa_Z$ to $0$.
For a floor with nonzero excess kurtosis the general share identity applies instead.
The Gaussian floor is a convenient model for small sensing errors and has zero fourth cumulant.
The result also holds for non-Gaussian floors with zero fourth cumulant and the stated finite
moments and independence.
\end{proposition}

The proof is given in Appendix~\ref{proof:o3inv}.

The recovery gain rescales the residual and its conditional standard deviation by the same
factor, so standardization removes that action. Physically, this stage suppresses new velocity
fluctuations without changing their conditional shape. Pooling the recovery episode retains
the different scales: a high-variance period contributes disproportionately to the fourth
moment. The independent scale-mixture formula quantifies this effect relative to a constant-scale
law, while comparison of two mixtures requires their respective scale-moment ratios.

Within-step cancellation changes the balance between disturbance and correction noise instead.
For fixed component laws and a zero-excess floor, increasing $\lambda$ lowers the disturbance's
variance share and hence its conditional excess kurtosis. Without sensing noise, partial
cancellation would only rescale the disturbance, just as the recovery gain does; complete
cancellation would leave zero variance and undefined kurtosis. Consequently, absence of a shape
change remains compatible with correction that acts through scaling alone.

A gain chosen in response to the current disturbance can treat large and small increments
differently. Remark~\ref{rem:restoration} identifies why this requires a joint moment calculation
rather than the predictable-scaling argument.

\Needspace{5\baselineskip}
\begin{remark}[Conditions for excess-kurtosis reduction]
\label{rem:restoration}
Predictability permits a scalar gain to be taken outside a conditional variance. If a gain
depends on the current disturbance, that calculation generally fails and its moments must be
derived from the joint law. The within-step linear mixing law sets conditional shape; the
predictable gain sets a variance response only after the covariance before recovery scaling is fixed.
\end{remark}

The instrument can introduce another independent component after the body has moved.
Under a suitable additive observation law, measurement noise can dilute excess kurtosis through
the same variance-share identity as correction noise. State propagation, filtering, and
reference fitting determine whether that identity applies to the measured residual, and changes
in the forcing mixture can also alter pooled shape. An observed reduction must therefore exceed
what the specified uncorrected motion and observation process predict. Remark~\ref{rem:attribution-scope}
separates this rejection of a benchmark from identification among the remaining mechanisms.

\Needspace{5\baselineskip}
\begin{remark}[What an observed excess-kurtosis reduction identifies]
\label{rem:attribution-scope}
A measured shape reduction beyond a specified prediction from uncorrected motion and measurement noise is evidence
against that prediction. Even jointly with a variance statistic, it does not identify a unique
cancellation mechanism without further restrictions on alternative input and observation models.
\end{remark}

Less variable motion indicates suppression only against the same forcing without correction.
Proposition~\ref{prop:o2o3} establishes that comparison. Shape alone leaves two explanations:
added body pulses and correction mixed with sensing noise can both make large increments less
prominent, as Lemma~\ref{lem:kurtshare} shows. A proportional gain can instead suppress
fluctuations without changing their conditional shape (Proposition~\ref{prop:o3inv}). Thus unchanged
kurtosis leaves that form of correction possible, while lower kurtosis alone leaves added forcing
possible. Both predictions must include the observing process
\citep{carpino2003errors,veres2017statistical}. When shape is unchanged or no usable forcing
contrast exists between event and quiet windows, an isolated impulse provides another comparison.

\subsection{Post-disturbance variance suppression}
\label{sec:o4}

A controller can strengthen its rejection of disturbances even when the environment remains
quiet and the conditional shape is unchanged. An isolated impulse makes this response testable:
it displaces the body, after which the same background fluctuations continue to act. Under the
inert restriction, the changed velocity does not alter their conditional covariance. Under the
rising-gain correction law, stronger rejection reduces the unpredictable part of later inputs.
The comparison thus varies the response to the incident impulse rather than requiring a
persistent increase in environmental forcing.

Fix the pre-impulse history $\mathcal H_0$ and compare two continuations: one receives an
exogenous velocity input $\bm\delta$ at $t_0$, and the other does not. The test input is distinct
from the stochastic innovation $\bm\eta_{t_0}$. Write $\mu$ for the gravitational parameter
used in the local orbital approximation. To distinguish the retained displacement from the
change in subsequent variability, define
\[
\begin{split}
\mu_h(\mathcal H_0,\bm\delta)&=
\E[e_{t_0+h}^{\bm\delta}\mid\mathcal H_0]-\E[e_{t_0+h}^{0}\mid\mathcal H_0],\\
v_h(\mathcal H_0,\bm\delta)&=
\E[\operatorname{tr}H_{t_0+h}^{\bm\delta}\mid\mathcal H_0]-
\E[\operatorname{tr}H_{t_0+h}^{0}\mid\mathcal H_0].
\end{split}
\]
Here $e$ is a fixed displacement coordinate and $h$ counts sampling steps of duration $\Delta t$.
Both continuations keep the environment quiet; the unperturbed (S) continuation has $K=0$
and covariance $\Sigma$ before recovery scaling. Using the same pre-impulse history compares
what the event changes, rather than comparing an event-selected record with an unrelated quiet
one. A population average uses the same distribution of histories on both sides, as in
generalized impulse-response analysis \citep{koop1996impulse,hafner2006volatility}.

With later mean inputs held equal, only the retained initial velocity separates the mean
continuations. A smaller observed ramp could therefore reflect either onset correction or a
smaller incident impulse. The additional question is whether later innovation variance changes
while background covariance stays fixed. Proposition~\ref{prop:o4} derives both responses,
showing how the variance contrast resolves the inert--rising-gain ambiguity left by the local
mean trajectory.

\Needspace{5\baselineskip}
\begin{proposition}[Local impulse ramp and conditional-variance response]
\label{prop:o4}
Use the same-history response definitions above, Assumption~\ref{ass:A2} in (I), and the linear
law and isolated-episode conditions of Assumption~\ref{ass:A3} in (S). In each regime,
the initial state is finite and fixed by $\mathcal H_0$, apart from the onset-retained
test impulse. At each future step $j$ in the response horizon, both intervention laws satisfy
\[
\E[\|\mathbf b_j\|+\|\bm\varepsilon_j\|\mid\mathcal H_0]<\infty,
\qquad
\E[\mathbf b_j^{\bm\delta}\mid\mathcal H_0]
=\E[\mathbf b_j^{0}\mid\mathcal H_0].
\]
Future innovations are conditionally centered under their respective intervention laws;
their covariances may differ. Conditional on $\mathcal H_0$ and the fixed test input,
$K(h)$ is a specified nonrandom schedule. Hold the environment quiet. In (I), its
conditional law given $\mathcal H_0$ is identical in both continuations, and at every
response horizon $h\ge0$ require
\[
\E[\operatorname{tr}g(W_{t_0+h})\mid\mathcal H_0]<\infty.
\]
In (S), the covariance before recovery scaling is the same fixed $\Sigma\succ0$ in both continuations.
Take the coordinate along
the nonzero input so $\delta=\|\bm\delta\|>0$, with the retained input collinear as stipulated
in Assumption~\ref{ass:A3}(b). Then:

\emph{(i) Level response (O4a).} In the constant-transition model~\eqref{eq:jordan},
$\mu_h^I=\delta h\Delta t$ and $\mu_h^S=\eta\delta h\Delta t$ for $h\ge0$.
These are local leading-order formulas under gravity, for small relative displacement and
$h\Delta t\ll\sqrt{r^3/\mu}$; they are not infinite-horizon Keplerian claims.
If the incident input size is unknown, the retained slope alone does not separate (I) and (S).

\emph{(ii) Variance response (O4b).} For the reaction-gain episode,
\[
v_h^I=0,\qquad v_h^S=[(1-K(h))^2-1]\operatorname{tr}\Sigma
=-K(h)(2-K(h))\operatorname{tr}\Sigma.
\]
The latter is zero at $K(h)=0$, strictly negative at $K(h)>0$, and nonincreasing along the
stipulated nondecreasing gain path. If that idealized path and covariance law are continued
indefinitely, its limit is $-K_{\max}(2-K_{\max})\operatorname{tr}\Sigma<0$.
This limit does not extend the physical validity horizon of part~(i). With the same constant
gain in both continuations, $v_h^S=0$ instead.
\end{proposition}

The proof is given in Appendix~\ref{proof:o4}.

After a kick, a changed course and a quieter response to later disturbances are separate
outcomes. Proposition~\ref{prop:o4} shows that both references have a smooth local mean displacement,
but only the rising-gain reference suppresses later velocity fluctuations at positive-gain
horizons. An unknown impulse size leaves the ramp compatible with a smaller uncorrected kick;
the variance response supplies the distinction. It concerns new increments, not automatic
narrowing of position uncertainty or return to the old path. Position accumulates those
increments, and return requires additional mean steering. The covariance result uses second
moments and scalar scaling, without requiring Gaussian disturbances or an exponential gain path.

This comparison detects a change in rejection after the impulse. A controller already operating
at the same constant gain in both continuations has no such change, even if it reduces each
new disturbance. Its effect must instead be measured against an uncorrected benchmark, through
an increment, shape, or variance-level contrast. Remark~\ref{rem:smodes} distinguishes these
two observation designs and the different matching information they require.

\Needspace{5\baselineskip}
\begin{remark}[Two observability modes of the stabilized regime]
\label{rem:smodes}
For the isolated impulse response, covariance before recovery scaling is fixed. A rising gain
produces a negative conditional-variance response wherever $K(h)>0$; the same constant gain in
both continuations produces zero response. Constant-gain stabilization can instead be detected
through attenuation relative to an unstabilized twin, or through a lower variance level
when correction outweighs the injected sensing noise.
The rising-gain response compares two continuations of the same object's history. The twin
and variance-level comparisons require a validated uncorrected benchmark, obtained from a model,
matched objects, or both. A comparison across objects must account for their coupling,
environment, and observation laws.
A fixed gain alone preserves conditional standardized kurtosis. Under the stated moment conditions,
mixing a positive-excess disturbance with an independent zero-excess sensing floor reduces
conditional excess kurtosis for $0<\lambda\le1$.
\end{remark}

An observed decline in variability can also occur because background forcing quiets or because the event detector selects the end of an active episode. Excluding the inert reference therefore requires the same-history, fixed-environment comparison of Remark~\ref{rem:eliminative}, together with its uncertainty requirement.

\Needspace{5\baselineskip}
\begin{remark}[Exclusion logic and alternative mechanisms]
	\label{rem:eliminative}
	A negative response estimates a departure from Assumption~\ref{ass:A2} only if the observation
	and event-selection model identifies the same-history, fixed-environment contrast.
	A statistical implementation could require a valid simultaneous confidence band below zero at
	pre-specified positive horizons where the alternative has $K(h)>0$, or a calibrated aggregate
	statistic. The onset response is zero and cannot satisfy a strict negative-band condition.
	A bootstrap label alone does not establish coverage or guarantee rejection.
	A valid rejection excludes the specified inert reference law. Attribution to stabilization
	additionally requires separating other admitted forcing and observation mechanisms; the
	application-specific checks are collected in Appendix~\ref{app:empirical}.
\end{remark}

Continued body activity can add further kicks after the initial impulse, producing more complicated
level responses through the delayed or interacting inputs represented by the finite-lag models in
Appendix~\ref{app:framework}. Weak, unresolved, or sufficiently regular pulses can nevertheless
produce a smooth record at the observed cadence, so the impulsive case in
Appendix~\ref{app:authority} requires its own forcing comparison.

At longer horizons, gravity alone can produce apparent return. Appendix~\ref{app:longhorizon}
derives the uncorrected continuation for each orbital class: period changes produce linearized
phase drift on bound orbits, while escape energy determines the growth of separation on
unbound paths. A crossing, drift, or bounded separation must therefore be compared with the
uncorrected trajectory and its uncertainties before it is attributed to steering.

For a specified pre-impulse forecast, one possible return statistic is
$d_{t_0,\tau}=\|\mathbf x_{t_0+\tau}
-\hat{\mathbf x}^{\,0}_{t_0+\tau\mid t_0^-}\|$.
A decrease or a finite-time zero is compared with the gravitationally propagated uncorrected
trajectory, keeping the reference fixed rather than refitting it after the impulse.

The operational question is whether the body merely follows a changed course or also changes
its response to new disturbances. Proposition~\ref{prop:o4} separates those outcomes for the
inert and rising-gain references. A controller that accepts the shifted course can still reject
disturbances (Remark~\ref{rem:adoption}), and a steady controller needs the matched comparisons
of Remark~\ref{rem:smodes}. Continued body activity supplies another possible cause of changing
variability. The joint comparison must therefore distinguish correction from both inert
propagation and additional forcing.

\subsection{Joint reduction and identifiable uniqueness}
\label{sec:joint}

The separate comparisons now explain why the diagnostics must be combined. Added body pulses
and correction noise can both lower excess kurtosis, so a shape reduction alone does not select
one mechanism. Their relation to the uncorrected twin differs: a pulsed body without correction
reproduces its own twin, whereas effective correction can suppress variance or shape relative
to that same forcing. With positive inert excess kurtosis and two nonzero independent channels obeying the stated
shape ordering, a match to the inert benchmark excludes the impulsive alternative; a strict
twin-suppression contrast excludes both uncorrected regimes. These comparisons first determine whether impulsive forcing remains
compatible, then whether the retained inert and stabilized explanations can be separated.

An isolated impulse supplies separation when the other moments remain alike. A rising recovery
gain changes later variance even if the environment does not become more variable and conditional
shape is preserved. Constant-gain correction instead requires a matched increment, shape, or
variance-level comparison. A joint signature records whichever of these physically justified
contrasts are available, rather than requiring every mechanism to produce the same diagnostic
pattern.

For a reference regime $R\in\{I,C,S\}$, define the available suppression contrasts
\[
d_{2,R}=\Delta V_R-\Delta H_R,\qquad
d_{3,R}=\kappa_R^{\mathrm{pass}}-\kappa_R,\qquad
d_{4,R}(h)=-v_h^R.
\]
The first compares event-minus-quiet trace-variance increments with the same
forcing's unstabilized twin. The second compares scalar excess kurtoses in a
fixed declared conditional law, using that law for both the object and its twin.
The third is the same-history impulse contrast of Proposition~\ref{prop:o4}.
An unavailable contrast supplies no separating evidence. Write $\mathsf F_R=1$ when
the centered forecast residual satisfies O1's maintained regularity conditions.

The sign of a population contrast is only part of the observational question. Weak correction
can make a strict suppression arbitrarily small, and projection or measurement noise can reduce
the difference further. To assess whether the observed contrasts resolve the classes, choose
$m\ge1$ common available continuous signature coordinates, a declared observation model, and
fixed positive scales $a_1,\ldots,a_m$. The dimensionless norm
$\|x\|_a=\max_{1\le j\le m}|x_j|/a_j$ permits variance and shape coordinates to be compared
without adding unlike physical units. Let $\mathcal S_R\subset\mathbb R^m$ be each regime's
nonempty set of allowed observation-level signatures, including permitted nuisance parameters.
The relevant separation is from the closest signature any rival regime can generate. For
$s\in\mathcal S_R$, define
\begin{equation}
\delta_R(s)=\inf_{u\in\bigcup_{R'\ne R}\mathcal S_{R'}}\|s-u\|_a.
\label{eq:resolution-gap}
\end{equation}
Let $\varepsilon\ge0$ be a fixed calibrated error radius for the estimated signature,
accounting for instrumental uncertainty, sampling, and uncertainty in the reference comparisons.

There are two distinct ways for the physical difference to become unresolved. Near-perfect
cancellation can leave only the same observable noise as vanishing forcing, as in
Lemma~\ref{lem:zeroresidual}. Alternatively, a residual can be detectable while its variance
and shape remain compatible with more than one mechanism. The first requires an explicit
amplitude threshold; the second requires separation of the joint signatures at the achieved
precision. Assumption~\ref{ass:resolution} keeps those requirements separate.

\Needspace{5\baselineskip}
\begin{assumption}[Detectable imperfect correction and regime separation]
\label{ass:resolution}
For an observational uniqueness claim, every admitted stabilized model leaves a tracking error
above the achieved residual-amplitude detection threshold:
\begin{equation}
\rho_S:=\|P(\mathbf s-\mathbf o)\|_{\mathcal I}>\tau_{\mathrm{obs}},
\qquad\tau_{\mathrm{obs}}>0.
\label{eq:tracking-floor}
\end{equation}
Here $(P\mathbf u)_t=P_t\mathbf u_t$, and $\tau_{\mathrm{obs}}$ is calibrated for the
reconstructed residual amplitude in the chosen observation units over $\mathcal I$.
Thus the corrective offset tracks the disturbance with an observable approximation error;
perfect and observationally unresolved cancellation are excluded from the admitted family.
For the true regime $R$ and signature $s\in\mathcal S_R$, additionally require
$\delta_R(s)>2\varepsilon$. This second condition separates the observable tracking signature
from the competing regimes; detectable tracking error alone does not establish uniqueness.
\end{assumption}

Finite sensing accuracy and actuator limits motivate imperfect correction, but they do not
ensure that its residual is detectable or that its signature is separated from every rival.
Those comparisons require the observation model and its calibrated uncertainty. When the
signature estimate can be displaced by at most $\varepsilon$, a rival must remain more than
another $\varepsilon$ away to be incompatible at that tolerance. This explains the sufficient
$2\varepsilon$ bound.

The separate contrasts must now distinguish all three mechanisms together.
Corollary~\ref{cor:joint} combines the shape restriction with suppression relative to uncorrected
motion, then applies the observation-level separation bound.

\Needspace{5\baselineskip}
\begin{corollary}[Identifiable uniqueness of the dynamical regime]
\label{cor:joint}
Consider the reference regimes of Definition~\ref{def:regimes}. Apply the hypotheses
of each diagnostic to the comparison in which it is used.
Parts~(i)--(ii) concern population signatures; part~(iii) additionally imposes the observable
imperfection and separation conditions of Assumption~\ref{ass:resolution}.

\emph{(i) Joint reduction.}
For a fixed scalar projection in a declared conditional law, suppose the inert
external-channel benchmark has $\kappa_I>0$. In (C), let the external and body
channels be independent, each with positive finite variance and finite fourth
moment, with $\kappa_e=\kappa_I$ and $\kappa_b\le\kappa_I$. Then
\[
\kappa_C<\kappa_I,\qquad d_{2,C}=d_{3,C}=0
\]
whenever the contrasts are defined. The joint O1--O3 condition
\begin{equation}
\mathsf F_R=1,\qquad
\bigl[\kappa_R=\kappa_I\ \text{or}\ d_{2,R}>0\ \text{or}\ d_{3,R}>0\bigr]
\label{eq:endpoint-reduction}
\end{equation}
therefore excludes (C) and is satisfied by (I). It also retains any (S) model
with an informative O2 or O3 suppression contrast. When $\lambda=0$ and
$\kappa_S^{\mathrm{pass}}=\kappa_I$, predictable positive recovery scaling
preserves $\kappa_S=\kappa_I$, so (S) is retained even without either suppression
signal. For candidates satisfying equation~\eqref{eq:endpoint-reduction}, the
reference set consequently reduces to $\{I,S\}$.

\emph{(ii) Separation of the retained pair.}
In (I), every defined contrast $d_{2,I}$, $d_{3,I}$, and $d_{4,I}(h)$ is zero.
In (S), each of the following gives a strict separation under its own hypotheses.

\emph{Variance attenuation (O2).} Under the strictness conditions of
Proposition~\ref{prop:o2o3}(iii), $d_{2,S}>0$. These include a positive uncorrected
event increment and effective within-step or recovery correction under the
stated common-window law.

\emph{Shape attenuation (O3).} In a fixed conditional law, suppose the twin has
positive excess kurtosis, the independent sensing floor has zero excess kurtosis,
both have positive finite variance and finite fourth moment, and $0<\lambda\le1$.
If $w$ is the disturbance's share of the mixed variance, then
\[
\kappa_S=w^2\kappa_S^{\mathrm{pass}},\qquad
d_{3,S}=(1-w^2)\kappa_S^{\mathrm{pass}}>0.
\]
Predictable positive recovery scaling preserves this conditional comparison.

\emph{Impulse-conditioned variance suppression (O4b).} Under
Proposition~\ref{prop:o4}, a reaction-gain episode gives
\[
d_{4,I}(h)=0,\qquad
d_{4,S}(h)=K(h)(2-K(h))\operatorname{tr}\Sigma>0
\quad\text{when }K(h)>0.
\]
This conclusion requires neither $\Delta V>0$ nor a positive excess kurtosis.
It therefore retains the I--S separation when the event-based variance and
shape comparisons are uninformative. The shared O4a responses
$\mu_h^I=\delta h\Delta t$ and $\mu_h^S=\eta\delta h\Delta t$ remain compatible
with a smooth local ramp when the incident impulse is unknown.

\emph{Constant-gain variance level.} A common constant gain gives
$d_{4,S}(h)=0$. For a matched forcing law with
$V=\operatorname{tr}\operatorname{Var}(\bm\zeta)>0$ and
$F=\operatorname{tr}F_\nu$, instead compare
\[
L_I=V,\qquad
L_S=(1-K)^2\bigl[(1-\lambda)^2V+\lambda^2F\bigr].
\]
Then $L_I-L_S>0$ if $\lambda=0$ and $K>0$, or if $\lambda>0$ and
$\lambda F<(2-\lambda)V$. O2 or O3 may also distinguish the constant-gain model
when their respective strictness conditions hold.

On a reference family in which the stabilized models satisfy the retention
condition in part~(i) and at least one strict suppression condition in part~(ii),
the joint population signatures identify a unique regime. O2 and O3 may each
contribute to both stages, while O4b supplies a separate endpoint comparison.

\emph{(iii) Observational uniqueness.}
Under Assumption~\ref{ass:resolution}, suppose the true signature is $s\in\mathcal S_R$
and its estimate satisfies $\|\widehat s-s\|_a\le\varepsilon$.
Then $R$ is the unique regime compatible with $\widehat s$, where compatibility means
$\inf_{u\in\mathcal S_R}\|\widehat s-u\|_a\le\varepsilon$.
If the error bound holds with probability at least $1-\alpha$, for $0<\alpha<1$,
the same uniqueness conclusion holds with probability at least $1-\alpha$.

\end{corollary}

The proof is given in Appendix~\ref{proof:joint}.

Under Corollary~\ref{cor:joint}'s restrictions, an inert classification means that two alternatives
have been excluded. Added body pulses would change the matched shape; the admitted corrected
models would suppress at least one response. Neither change is present in the inert signature.
The impulsive reference instead changes shape while still matching its own uncorrected twin;
the admitted stabilized references show a strict suppression contrast. These combinations
identify what the body does to disturbances even when its individual parameters remain unknown.

A detectable departure still needs an identifiable cause. Part~(iii) requires enough separation
that measurement error cannot make an admitted rival compatible: a gap greater than twice the
calibrated radius suffices. The bound covers every permitted rival specification, rather than
only the example trajectories selected for illustration.

The remaining statistical question is whether an estimated signature recovers these population
distinctions as the record becomes informative. Distance to a regime set changes by no more
than the error in the signature itself. Thus a positive gap to incompatible regimes eventually
outweighs estimation error when the signature estimate converges. Corollary~\ref{cor:classconsistency}
uses this property to obtain class consistency for a singleton and convergence in class
membership when several regimes share the same signature.

\Needspace{5\baselineskip}
\begin{corollary}[Class consistency and convergence to a class set]
\label{cor:classconsistency}
For $x\in\mathbb R^m$, write
\[
D_R(x)=\inf_{u\in\mathcal S_R}\|x-u\|_a,
\]
and let a measurable signature classifier choose
$\widehat R_T\in\arg\min_R D_R(\widehat s_T)$.
Suppose the true signature $s$ belongs to $\mathcal S_{R_0}$ and
$\widehat s_T\to s$ in probability. Let $\mathcal R(s)$ be the compatible class set in
equation~\eqref{eq:signature-class-set}. If
\[
\gamma(s):=
\min_{R\notin\mathcal R(s)}D_R(s)>0
\]
whenever the complement of $\mathcal R(s)$ is nonempty, then
\[
\Pr\!\left(\widehat R_T\in\mathcal R(s)\right)\longrightarrow1.
\]
If $\mathcal R(s)=\{R_0\}$, the classifier is consistent for the true dynamical
class: $\Pr(\widehat R_T=R_0)\to1$. If instead $\mathcal R(s)=\{I,S\}$,
the conclusion retains the pair without identifying a unique member.
A consistently estimated additional coordinate that separates the true signature
from the rival regime by a positive gap gives class consistency for the enlarged signature.
For either signature, almost-sure convergence of its estimator makes the corresponding
class-membership conclusion hold eventually almost surely.
\end{corollary}

The proof is given in Appendix~\ref{proof:classconsistency}.

A longer or more precise record helps when it measures a genuine difference between mechanisms.
If two classes predict the same selected features, measuring those features more accurately
still leaves both possible. Corollary~\ref{cor:classconsistency} formalizes this distinction:
with a positive gap to excluded regimes, convergent estimates recover the class or remain
within the compatible class set. A further feature can separate a shared signature. For example,
a consistently estimated rising-gain impulse response can distinguish inert and stabilized
motion when event-based moments are uninformative, provided the enlarged signature has a
positive gap.

That extension must still address every admitted rival. The first reduction requires positive
inert excess kurtosis and two nonzero independent channels in (C). Admitting a purely body-driven
object requires another justified contrast to exclude it; O4b alone separates only its specified
inert and stabilized references. Weak correction or a small body-variance share can also make
a population difference arbitrarily small. Enlarging the forcing or observation family may
therefore require either better measurement of an existing difference or another identifying
feature. The sufficient finite-resolution bound in Corollary~\ref{cor:joint}(iii) guarantees
unique compatibility; a particular record can still have a single compatible class when that
bound is unmet.

\subsection{Illustrative signatures and screening}
\label{sec:illustration}

The formal results imply observable differences in how the reference regimes respond to disturbances.
Figure~\ref{fig:regimes} shows those differences at one illustrative operating point.

\noindent\begin{minipage}{\textwidth}
	\captionsetup{type=figure}
	\centering
	\includegraphics[width=\linewidth]{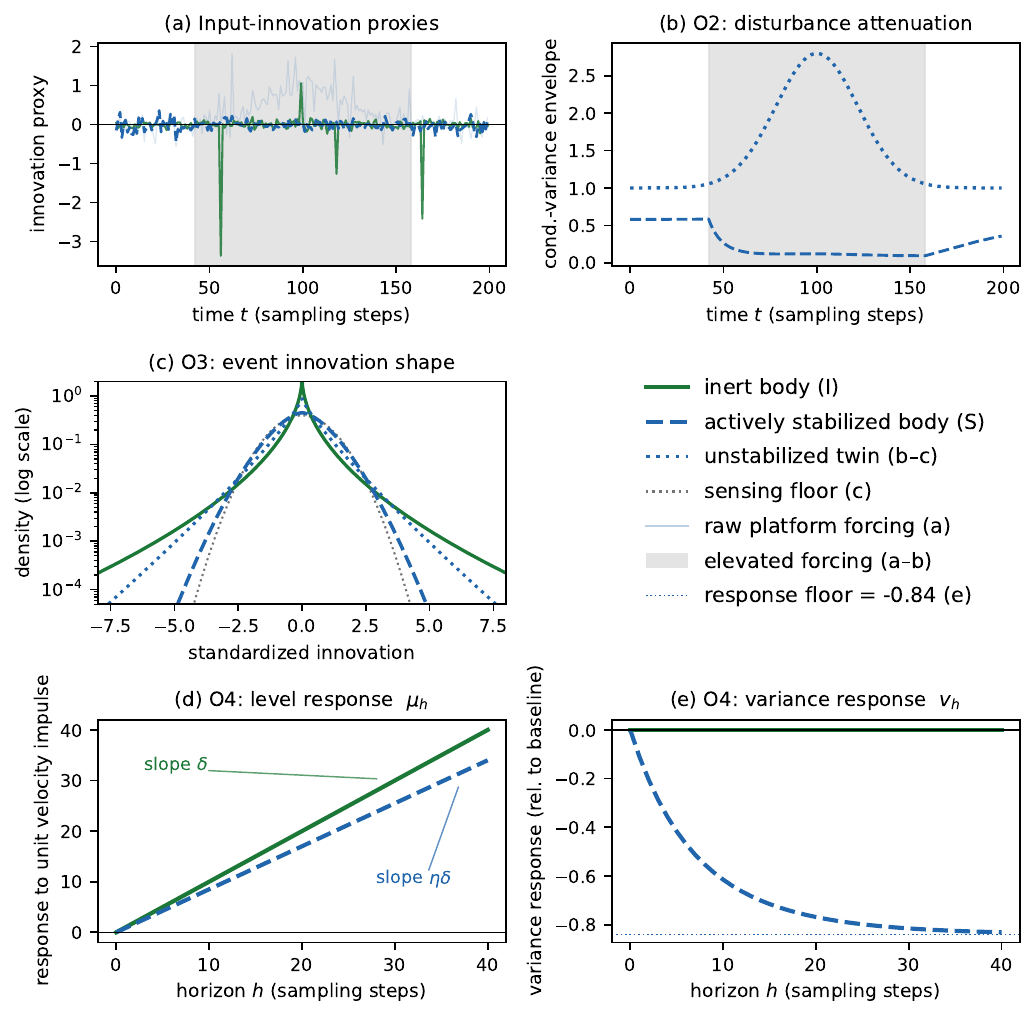}
	\caption{\textbf{Illustrative inert and stabilized reference models.}
		Green denotes inert, blue stabilized, with the unstabilized twin and sensing floor where indicated.  
		(a) Centered input-innovation draws used as illustrative forecast-residual proxies under time-varying disturbance scales; the pale forcing curve also contains a deterministic mean. O1 applies to the output of a specified observer.  
		(b) Conditional variance attenuation relative to the same forcing's twin.  
		(c) Standardized generalized-error densities illustrating conditional shape; the stabilized curve matches the excess kurtosis implied by the share identity with a Gaussian sensing floor.  
		(d) Local retained-impulse displacement ramps and (e) variance responses with covariance held fixed before recovery scaling, as in Proposition~\ref{prop:o4}. Parameters are $\eta=\figEta$, $K_{\max}=\figKmax$, $\tau=\figTau$, and $\lambda=\figLambda$, with $\tau$ the gain timescale in sampling steps.  
		The stochastic panels use fixed seeds; (c)--(e) use the stated analytic laws.}
	\label{fig:regimes}
\end{minipage}
\par\medskip

\pagebreak

The panels collect the physical contrasts behind identification. O1 supplies the shared
forecast-error regularity assumed for the inert and stabilized models. Under the stated
comparison laws, correction attenuates the additional variability that stronger forcing
produces in an uncorrected body (O2). The interaction between correction and sensing noise
can also change conditional shape (O3). After an isolated impulse, the inert and stabilized
references can share a smooth displaced course (O4a), while the rising-gain stabilizer
increasingly suppresses new velocity fluctuations (O4b). Impulsively forced motion admits a
wider range of post-disturbance responses, depending on the timing and strength of later
body-generated inputs (Remark~\ref{rem:cometary}).

Screening compares a trajectory's estimated signature with the full signature sets admitted
by the candidate regimes. Corollary~\ref{cor:classconsistency} establishes consistency of
nearest-signature classification when the true signature is separated from the rival sets
and the signature estimate converges. Corollary~\ref{cor:joint}(iii) gives the corresponding
sufficient separation condition at a stated estimation tolerance. The physical comparisons
in parts~(i)--(ii) organize the resolved reduction as
\[
\{I,C,S\}
\;\xrightarrow{\text{joint regularity and moment evidence}}\;
\{I,S\}
\;\xrightarrow{\text{resolved correction contrasts}}\;
\{I\}\ \text{or}\ \{S\}.
\]
Regularity alone leaves the impulsive alternative open; the first reduction requires the
joint moment evidence. The second separates transmitted from suppressed disturbances using
the comparisons appropriate to the admitted correction laws. Unresolved contrasts leave
the corresponding classes together.

The general screening rule follows directly from the signature-set formulation:

\begin{center}
	\fbox{\begin{minipage}{\dimexpr\linewidth-2\fboxsep-2\fboxrule\relax}
			\textbf{Example algorithm: signature-based screening.}
			\begin{enumerate}[label=\arabic*.,leftmargin=*]
				\item Specify the candidate regimes and their forcing, observation, and reference-comparison laws,
				and construct the corresponding signature sets $\mathcal S_R$ from the available contrasts.
				\item Estimate the observed signature $\widehat s$ and calibrate its joint error radius
				$\varepsilon$, including measurement uncertainty and uncertainty from reference fitting
				and event selection.
				\item Evaluate $D_R(\widehat s)$ for each candidate regime and retain every regime satisfying
				$D_R(\widehat s)\le\varepsilon$.
			\end{enumerate}
	\end{minipage}}
\end{center}

A single retained regime gives a unique compatible class under the stated comparison.
Several retained regimes define the alternatives that further observations must distinguish.
An empty set calls for review of the admitted models and uncertainty calibration.
This general rule admits different computational implementations; the reduction structure
of the present three-class problem suggests a constrained classification tree that organizes
the comparisons in Step~3.

The theory identifies the relevant features, class partitions, and directions of comparison.
Numerical thresholds depend on the observing program. They could be calibrated from the
observation model or learned from a labelled corpus of independently classified objects,
while preserving the reduction order above. Training, calibration, and evaluation would
use separate objects, with transfer to new instruments or environments validated independently.
Appendix~\ref{app:empirical} collects the benchmark, event-selection, and measurement
requirements for this calibration.

\Needspace{20\baselineskip}
\refstepcounter{algorithm}\label{alg:classification}
\begin{center}
	\fbox{\begin{minipage}{\dimexpr\linewidth-2\fboxsep-2\fboxrule\relax}
			\textbf{Algorithm~\thealgorithm. Illustrative theory-guided screening tree.}
			\begin{enumerate}[label=\arabic*.,leftmargin=*]
				
				\item \emph{Establish and calibrate the comparisons.}
				Specify the observer, matched forcing benchmarks, and measurement law. Assess O1 regularity
				and, where available, the O4a impulse comparison. Estimate the available contrasts in
				$\widehat s$ and calibrate their joint error radius $\varepsilon$, including uncertainty in
				the reference comparisons. Fix any learned thresholds before evaluating the target object.
				
				\item \emph{Reduce the impulsive alternative.}
				Under the hypotheses of Corollary~\ref{cor:joint}(i), the admitted (C) models have
				$\kappa_C<\kappa_I$ and $d_{2,C}=d_{3,C}=0$. Use the joint O1--O3 evidence to exclude
				(C) when its allowed shape and attenuation pattern is incompatible with the observed
				signature at the calibrated tolerance; otherwise retain it.
				
				\item \emph{Separate the inert and stabilized candidates.}
				Use the O2, O3, or O4b suppression contrasts of Corollary~\ref{cor:joint}(ii), including
				the matched variance-level comparison for constant gain. Retain (I) until its
				zero-suppression pattern is excluded, and retain (S) until all its admitted correction
				signatures are excluded. Where the available contrasts do not resolve the pair,
				retain both candidates.
				
				\item \emph{Return the jointly compatible classes.}
				Confirm the branch decisions against every regime's full joint signature set, requiring
				each candidate to account for the accumulated features under a common admissible
				specification. Retain exactly the regimes with $D_R(\widehat s)\le\varepsilon$.
				Report a class when it is the sole survivor; otherwise report the compatible set,
				or flag model and calibration review if the set is empty.
				
			\end{enumerate}
	\end{minipage}}
\end{center}

The branches follow the reduction in Corollary~\ref{cor:joint}, and the final distance
check preserves the joint compatibility criterion. Statistical tests, confidence regions,
or learned thresholds can guide the branch decisions. Their calibration must account
for dependence among diagnostics and the sequence of comparisons. The corollary's
resolvability guarantee applies under its stated joint separation and error-bound conditions.

The Monte Carlo experiments in Appendix~\ref{app:mc} show where these component
comparisons lose resolution. In the simulated designs, outliers distort regularity-related
statistics, short arcs can produce wrong-sign kurtosis gaps, and pooling changing gains
can obscure conditional shape. The O4b comparison is sensitive to measurement noise,
environmental drift, and event selection; calibration must reproduce the event-selection
procedure actually used. Extending the arc adds limited information when the post-onset
response window stays fixed, whereas constant-gain discrimination improves with longer
records in the matched variance-level simulations. Applying the latter comparison also
requires a credible ballistic benchmark. These experiments assess the constituent
diagnostic comparisons. Evaluation of the complete classifier, including any learned
thresholds, would additionally track false exclusions, unresolved branches, and final
class assignments on held-out objects. The observed sensitivities identify where a
screening application would benefit from better-matched references, stronger calibration,
or additional observations.

\section{Discussion and conclusion}
\label{sec:discussion}

An acceleration anomaly reveals a departure from a specified trajectory, but several physical
mechanisms can account for the same departure. The question is whether their responses to
disturbances differ enough to identify the mechanism. Within the reference families considered
here, comparing those responses under common forcing and observation conditions gives a joint
signature that can identify the dynamical class, even when individual forcing strengths and
correction settings remain undetermined. Corollary~\ref{cor:joint} establishes the separating
comparisons and their finite-resolution guarantee; Corollary~\ref{cor:classconsistency} shows
how consistent signature estimation recovers the class when the population gap is positive.

Studies of 3I/ATLAS illustrate the physical information available from different observations.
Combining an astrometric upper bound on acceleration with inferred outflow properties gives
model-dependent lower bounds on nucleus mass and diameter \citep{cloete2025limits}. Jet
morphology and photometry constrain rotation, with a small period discrepancy plausibly
attributable to systematics and aliasing \citep{scarmato2026rotation}. Sublimating-grain models
account for the sunward antitail and evolving coma brightness
\citep{keto2025antitails,keto2025icecoma}. Linking jet geometry and outflow estimates to
acceleration supplies a further physical constraint \citep{classify2604}. Each comparison
constrains part of a proposed explanation: the body's properties, its activity, or the recoil
that activity produces.

A fit to one of these features establishes compatibility with that feature; identification
requires separating the joint predictions of competing mechanisms. Models may reproduce similar
average acceleration while predicting different fluctuations or responses to an impulse. Applying the present criterion to
3I/ATLAS would require those joint predictions, their observation laws, and calibrated uncertainty
for the actual record. A technological interpretation, such as that proposed by
\citet{hibberd2025technology}, requires the same connection between its proposed mechanism and
observable behavior. Specifying an origin alone supplies no correction law. The distinction to
be tested is how a proposed mechanism transmits, adds, or offsets disturbances, with independent
physical evidence constraining which mechanisms are plausible.

What position data can establish therefore depends on the reference family
\citep{rothenberg1971identification}. An inert classification concerns propagation without
pulsed body inputs or correction at the diagnostic scale; a spacecraft coasting without
correction can fall within that class. Composition and origin require further physical evidence.
The forcing and correction restrictions also need justification for the particular object.
A smooth radiation-pressure law describes mean forcing but leaves its variability to be
specified, and continuing outgassing requires a law for later inputs beyond an isolated impulse.
Nonlinear or time-varying dynamics may change the comparisons. Alternative distributions retain
the moment-based contrasts when they satisfy the relevant independence, moment, and covariance
conditions; resemblance to an illustrative curve is insufficient. Table~\ref{tab:mimics}
identifies alternative mechanisms and observations that could help distinguish them.

The observing process can further obscure a physical difference between mechanisms. Optical
astrometry measures sky-plane angles, so the comparison must account for projection, range
uncertainty, and the components that the record constrains. For scale, the ground-based precision
range of 0.1--1 arcsecond used here \citep{carpino2003errors,veres2017statistical} corresponds,
at one astronomical unit, to approximately 73--725 km in the sky plane. These single-measurement
scales are not absolute limits on inference from the full record: repetition and a well-characterized
error law can distinguish signals that are weak in individual observations. Dependence and
uncertain reference models can limit that gain. Better
post-event astrometry can itself create apparent quieting. Irregular cadence, reference fitting,
and event selection must therefore enter the calibration, so that an estimated suppression can
be distinguished from a change in how the body was observed
(Appendix~\ref{app:empirical}).

Even precise observations need a defensible uncorrected comparison. An unstabilized twin
describes how the same forcing would move the object without correction; a different body's
record can approximate it only after differences in physical coupling, environment, and
observation have been accounted for. A validated ballistic cohort may support a variance-level
comparison, but an unusually quiet candidate is not evidence of correction merely because the
other objects are more variable. The isolated-impulse comparison instead uses a continuation of
the candidate's own history. It avoids that cross-object comparison, but requires an identified
disturbance and a justified model for the environment afterward. The useful diagnostic is thus
determined partly by the available reference information, rather than by the trajectory's length
or apparent quietness alone.

The simulations in Appendix~\ref{app:mc} show how these differences in comparison design affect
observing strategy. Extending the total arc offers limited improvement for the event statistic
when its post-onset response window remains fixed. A matched variance-level comparison can use
observations throughout the record and benefits from a longer arc in the illustrative experiments.
For an observing program, these results motivate attention to when measurements are collected
and which comparison they inform, alongside their total number. Measurement noise reduces
sensitivity, and an estimated excess-kurtosis contrast can have the wrong sign on an individual
arc. Its uncertainty needs assessment jointly with the other evidence. The reported rejection
frequencies describe the simulated laws; object-specific calibration determines how informative
the corresponding comparison is for an actual record.

Combining trajectory analysis with independent physical measurements can constrain alternatives
that positions alone leave unresolved. Photometry and spectroscopy restrict outgassing and
composition \citep{mumma2011composition,jewitt2022interstellar}. Orbit determination extracts
mean-motion evidence through Marsden-type parameters \citep{marsden1973nongrav}, detections of
non-gravitational acceleration \citep{micheli2018oumuamua}, and residual error models
\citep{carpino2003errors,veres2017statistical}. The disturbance-response comparisons add
restrictions on how a proposed mechanism transmits or offsets fluctuations. In a joint analysis,
inferred outflow and jet geometry could restrict the strengths and directions of body-generated
forcing, while the trajectory tests whether those inputs account for the measured mean motion
and variability. Uncertainty in the physical relations remains part of the predicted signature
sets. The gain from combining observations is that the same physical explanation must satisfy
all of them.

Independent physical information can also constrain a proposed correction mechanism. Once the
disturbance, mass, and command duration are specified, the actuator's reachable set limits the
correction it can supply (Appendix~\ref{app:authority}). Feasibility within that set establishes
what is possible; the disturbance response provides evidence about its use. Finite saturation
changes the moment law when the requested correction exceeds the available command. When
hardware and policy are unknown, the comparison must admit the corresponding range of correction
laws. Physical constraints and trajectory statistics then test different parts of the explanation:
whether the proposed correction could be executed and whether its predicted response agrees
with the record.

A classification supported by this combined evidence can complement frameworks for deciding
which objects merit further investigation. The Loeb Scale assigns significance levels and
response protocols \citep{eldadi2025scale}, and its evolving formulation updates an effective
score as observations accumulate \citep{trivedi2025evolving}. \citet{szocik2026threshold}
interpret candidate-technosignature status as a reason to intensify investigation while origin
remains unresolved. An object can warrant that attention even when several physical explanations
are compatible. The compatible class set adds a different kind of information: it specifies
which explanations remain and which predicted responses would distinguish them. A follow-up
program can use the priority ranking to decide which objects to observe and the compatible
class set to decide which differences those observations should test.

The example algorithm in Section~\ref{sec:illustration} provides one way to connect that
information to observing choices. A proposed observation can be assessed by the difference it
would resolve between the remaining predictions, at the precision and cadence expected. The
proposal to observe 3I/ATLAS more closely with Juno \citep{loeb2025juno} illustrates the type of
follow-up for which this question matters. Coordinated characterization and targeted observation
programs \citep{eldadi2025global,trivedi2026network} could use the same comparison when allocating
further astrometry, activity measurements, or spectroscopy.

The remaining ambiguity determines which evidence is needed. More precise trajectory
measurements can resolve a difference already predicted by competing mechanisms. Where those
mechanisms share the selected trajectory features, activity or composition measurements can
provide another identifying restriction. A joint analysis can thus recover the dynamical class
where the evidence separates it and direct further observations toward the physical questions
that remain.

\section*{Data and code availability}

The numerical experiments use synthetic data; no observational or experimental data were
generated in this study. Observational results cited for motivation and the astrometric-precision
scale in Section~\ref{sec:discussion} come from the published sources identified in the text.
Code to reproduce the results and figures is archived on Zenodo at
\url{https://doi.org/10.5281/zenodo.22876267}.

\clearpage
\appendix

\section{Notation}
\label{app:notation}

Table~\ref{tab:notation} lists the displacement residuals, velocity innovations, correction
parameters, and identification quantities used throughout the analysis.

\noindent\begin{minipage}{\textwidth}
\captionsetup{type=table}
\centering\small\singlespacing
\begin{threeparttable}
\caption{\textbf{Recurring notation.}}
\label{tab:notation}
\begin{tabularx}{\textwidth}{l >{\raggedright\arraybackslash}X l >{\raggedright\arraybackslash}X}
\toprule
Symbol & Meaning & Symbol & Meaning \\
\midrule
$\mathcal R,\ R_0$ & candidate regimes; true regime & $\theta,\ \Theta_R$ & nuisance parameters; their regime-specific space \\
$s,\ \widehat s_T$ & population signature; its estimate & $\mathcal S_R$ & signatures admitted by regime $R$ \\
$\mathcal R(s)$ & regimes compatible with signature $s$ & $\varepsilon$ & calibrated signature error radius \\
$\mathbf{e}_t,\ \mathbf y_t$ & spatial residual; measured residual & $\lambda$ & within-step sensed-cancellation share, $\lambda\in[0,1]$ \\
$\mathbf{r}_t$ & input to the recovery stage, eq.~\eqref{eq:withinstep} & $t_0$ & disturbance (impulse) onset time \\
$\bm{\zeta}_t$ & velocity innovation before correction & $K_t,\ K_{\max}$ & predictable rejection gain; its ceiling \\
$\bm{\nu}_t$ & loop-injected velocity noise & $\tau$ & gain timescale; also a return horizon \\
$W_t$ & predictable forcing-environment index & $\eta$ & retained test-impulse fraction, scalar; distinct from $\bm{\eta}_t$ \\
$g(W_t),\,G(W_t)$ & inert input covariance $C+G(W_t)$; its variable part & $\delta$ & magnitude of the test velocity impulse at $t_0$ \\
$H_t$ & velocity-innovation conditional covariance & $u_0,\ u_{\max}$ & thrust operating point; its ceiling \\
$\Sigma$ & quiet covariance before recovery scaling & $\varrho(u_0)$ & inradius for independent thrust deviations \\
$\mu_h,\ v_h$ & level and conditional-variance responses & $\bm{\varepsilon}_t,\ \bm{\eta}_t$ & corrected velocity innovation, eqs.~\eqref{eq:withinstep}--\eqref{eq:authority} \\
$\kappa$ & scalar standardized excess kurtosis & $\chi$ & reference tidal indicator, $\sup\sqrt{\mu/r^3}\,T_{\mathrm{obs}}$ \\
$\mathbf{q}_t,\ \Phi$ & centered forecast residual; lag-one projection & $\mathbf{s}_t,\ \mathbf{o}_t$ & length-valued forcing and correction contributions \\
$\fext_t,\ \fbody_t$ & displacement contributions of smooth-mean and pulsed inputs & $R_t,\ \Mcal$ & measurement covariance; optional input-adjustment block \\
\bottomrule
\end{tabularx}
\begin{tablenotes}[flushleft]\footnotesize
\item Notes: regimes are (I) inert, (C) impulsively forced, (S) stabilized; the two uncorrected regimes
(I) and (C) have no corrective offset $\mathbf{o}_t$, unlike the actively
corrected (S). Estimators carry hats
($\hat v_h$); event-window increments carry $\Delta$ ($\Delta H$, $\Delta\kappa$).
\end{tablenotes}
\end{threeparttable}
\end{minipage}
\par\medskip

\clearpage
\section{Long-horizon trajectory classes}
\label{app:longhorizon}

A retained velocity impulse initially produces a displacement ramp, but gravity bends both the
perturbed and reference trajectories. An apparent slowing, crossing, or return can therefore
arise without corrective thrust. Assessing steering requires the uncorrected motion expected
from the same initial impulse. In the two-body reference model, orbital energy determines that
continuation. For a bound reference orbit, let $a$ be the semimajor axis, $e$ the eccentricity,
and $n=\sqrt{\mu/a^3}$ the mean motion. An impulse that changes the period produces accumulating
phase error in the linearized displacement, even though two exact bound trajectories remain
within finite radii. Hyperbolic trajectories instead approach constant outgoing velocities, whose
difference determines their leading relative motion. At the parabolic boundary, the relative
growth rate depends on which terms remain after subtracting the two trajectories.

Corollary~\ref{cor:classnull} gives these uncorrected displacement laws and the short-horizon
conditions for a local ramp. Assessing return or steering requires this propagated comparison,
with the gravity model and its uncertainty specified for the observed arc.

\Needspace{5\baselineskip}
\begin{corollary}[Impulse response without correction by trajectory class]
\label{cor:classnull}
Compare a reference trajectory and a trajectory receiving a single velocity impulse
$\bm\delta$ at the same initial position, with no later non-gravitational input. Keep the
pre-impulse reference fixed. For gravity-dominated comparisons assume Newtonian two-body motion
with $\mu>0$, nonzero angular momentum and no collision. The instantaneous radial and transverse
directions are respectively $\hat{\mathbf r}=\mathbf r/\|\mathbf r\|$ and
$\hat{\mathbf h}\times\hat{\mathbf r}$, where $\hat{\mathbf h}$ is the reference unit angular-momentum vector;
the transverse direction is not the velocity direction on an eccentric orbit.

\emph{(i) Local free-flight approximation.} The exact double-integrator model has
$\Delta\mathbf r(T)=T\bm\delta$. On a sufficiently short Keplerian window it is the
first-order local approximation, controlled by the gravity-gradient bound and the smallness
of the relative displacement. The dimensionless indicator
$\chi=T_{\mathrm{obs}}\sup_{\mathrm{window}}\sqrt{\mu/r^3}$ is a reference-orbit screening
scale, not a complete error bound for arbitrarily large perturbations.

\emph{(ii) Bound reference, $0\le e<1$.} In first-order variation about the reference ellipse,
$\delta a=(2a^2/\mu)\mathbf v\cdot\bm\delta$ and
$\delta n=-3n\delta a/(2a)$. The cross-track displacement is bounded and periodic. The
in-plane variation is bounded and periodic exactly when $\delta a=0$; otherwise its transverse and radial
secular envelope rates are
\[
R_T=\tfrac32 n|\delta a|\sqrt{\frac{1+e}{1-e}},\qquad
R_R=\tfrac32 n|\delta a|\frac{e}{\sqrt{1-e^2}},\qquad
R_R/R_T=\frac{e}{1+e}.
\]
For $e=0$, $R_R=0$. These are rates in the linearized solution, valid while accumulated
phase error remains small. If both exact post- and pre-impulse trajectories remain bound,
their exact separation is bounded for all time, even when their orbital periods differ.
Exact energy neutrality is $\mathbf v\cdot\bm\delta+\|\bm\delta\|^2/2=0$, not merely
$\mathbf v\cdot\bm\delta=0$.

\emph{(iii) Unbound reference.} If both trajectories escape hyperbolically with outgoing
velocity vectors $\mathbf v_\infty^1,\mathbf v_\infty^0$, then
\[
\Delta\mathbf r(t)=(\mathbf v_\infty^1-\mathbf v_\infty^0)t+O(\log t).
\]
Thus a nonzero difference of outgoing velocities gives linear divergence; when that
difference is zero, the displayed conclusion is only the remainder bound $O(\log t)$.
If one trajectory is hyperbolic and the other is
bound or parabolic, relative displacement divided by $t$ tends to the signed hyperbolic outgoing
velocity and separation grows linearly. For a parabolic reference, an energy-removing impulse
that makes the perturbed trajectory bound leaves its separation from the still-parabolic reference
asymptotic in norm to $(9\mu/2)^{1/3}t^{2/3}$. If both remain parabolic, separation is $O(t^{2/3})$; its leading term vanishes when
the periapsis directions agree, as obtained by subtracting the outgoing expansions in the proof. Bounded separation, a crossing, recurrence and convergence are distinct properties.
\end{corollary}

The proof of Corollary~\ref{cor:classnull} is given in Appendix~\ref{proof:classnull}.
Gravity can bend the displaced course without any corrective command. Two bound orbits can remain within finite separation while their phases differ;
different escape velocities instead produce sustained separation. A crossing or bounded
separation therefore becomes evidence about steering only through comparison with the
uncorrected motion expected from the same impulse. These are statements about the level path,
not the conditional-variance suppression measured by O4b.

The local ramp is useful before differential gravitational curvature materially changes that
level comparison. Let $L$ denote the supremum of the gravity-gradient operator norm along the
reference orbit over the observing window. Remark~\ref{rem:kepler} relates this bound to elapsed
time, giving a screening scale while retaining the requirement of small relative displacement.

\Needspace{5\baselineskip}
\begin{remark}[When the local ramp approximation applies]
\label{rem:kepler}
The dimensionless reference indicator $\chi$ is a screening scale for the local linearized
ramp. For a Kepler reference the operator-norm gravity-gradient bound satisfies
$LT_{\mathrm{obs}}^2=2\chi^2$. A small value controls the linearized remainder; a nonlinear
comparison also needs the perturbed path to remain close to the reference and away from collision.
The supremum records peak tidal strength over the window; its duration is a separate feature of
the orbit.

The local rate $\sqrt{\mu_\odot/r^3}$ is $\chiRateAU\,\mathrm{s^{-1}}$ at 1 au.
Holding radius fixed for a scale comparison gives $\chi=1$ after $\chiDaysOneAU$ days
at 1 au and $\chiDaysOneFiveAU$ days at 1.5 au; these are not exact phases swept on an
eccentric passage. Whether an orbit is bound is determined by energy, not by this indicator.
At intermediate horizons, propagate the uncorrected orbital model with uncertainty rather than
requiring a full period or treating $\chi$ as a binary validity theorem.
\end{remark}

Once the uncorrected continuation is known, return becomes a question about the controller's
chosen reference. Rejecting new disturbances can stabilize motion around a displaced path;
returning to the original path requires additional mean commands. Remark~\ref{rem:return}
separates that steering question from the local variance response, preventing gravitational
recurrence or sustained displacement from being assigned a control interpretation on its own.

\Needspace{5\baselineskip}
\begin{remark}[Long-horizon return]
\label{rem:return}
Long-horizon return toward a reference is a possible property of a specified steering policy,
not a consequence of the local gain and onset-retention assumptions.
Return is assessed by comparison with the propagated initial-condition null and its uncertainties.
Adopting a shifted reference is one possible reason a controller has no return to the old
reference, but is not the only observational ambiguity.
Proposition~\ref{prop:o4} proves a local ramp and a conditional covariance response; it does
not prove a general long-horizon separation between physical classes.
\end{remark}

\clearpage
\section{Actuator authority and post-disturbance response}
\label{app:authority}

The linear correction law assumes that the required opposing command can be executed. A finite
actuator can remove only a limited velocity increment over a given time, and this limit depends
on direction. The coplanar geometry of Figure~\ref{fig:authority} makes that physical constraint
explicit: balanced thrust has zero net force, while thrust deviations generate a bounded set of
directional commands. Comparing the correction required by the stochastic law with this
reachable set distinguishes a feasible command from a response that would require saturation
and a different moment calculation.

\subsection{Reachable sets and realized commands}

Three coplanar directions represent the in-plane authority of the four-thruster
architecture of \citet{andree2026thruster}. Out-of-plane control requires attitude dynamics,
which are outside this static model.
Let $u_0$ be the common baseline thrust and $u_{\max}$ the maximum thrust of each actuator,
with $0\le u_0\le u_{\max}$. For three coplanar unit directions
$\mathbf d_1,\mathbf d_2,\mathbf d_3$ at $120^\circ$, $\sum_i\mathbf d_i=0$.
Write $\Pi=\operatorname{span}\{\mathbf d_1,\mathbf d_2,\mathbf d_3\}$ for their plane. The full unipolar set
$\{\sum_i u_i\mathbf d_i:0\le u_i\le u_{\max}\}$ is independent of the nominal common
baseline. With symmetric independent deviations $|d_i|\le a$ about that baseline, where
$a=\min\{u_0,u_{\max}-u_0\}$, the set
\[
\mathcal U_a=\left\{\sum_{i=1}^3d_i\mathbf d_i:|d_i|\le a\right\}
\]
is, for $a>0$, the regular hexagon with vertices $\pm2a\mathbf d_i$.
A linear image of the coefficient cube is the convex hull of its eight vertex images;
two images are zero and the other six are $\pm2a\mathbf d_i$.
For $a=0$ the set is $\{0\}$. Adjacent nondegenerate vertices have central angle $\pi/3$, so the inradius relative to $\Pi$ and the maximum norm are
\[
\varrho(u_0)=2a\cos(\pi/6)=\sqrt3 a,\qquad
\max_{u\in\mathcal U_a}\|u\|=2a.
\]
This calculation permits independent deviations. If the additional constraint
$\sum_i d_i=0$ of the constant-total-thrust construction in \citet{andree2026thruster}
is imposed, the coefficient-section vertices are permutations of $(a,0,-a)$; their images
have norm $\sqrt3 a$, and the inradius is $3a/2$ instead. The two constraints therefore give different reachable sets. A policy determines the mean
command and realized thrust reductions; symmetry of the feasible set alone determines neither.

To apply this force constraint to the trajectory model, the required correction must be
expressed in velocity units. The controller removes the difference between the uncorrected
innovation and the innovation remaining after both correction stages. For the linear law,
that removed increment is
\[
\mathbf a_t=\bm\zeta_t-(1-K_t)[(1-\lambda)\bm\zeta_t+\lambda\bm\nu_t]
=[\lambda+K_t(1-\lambda)]\bm\zeta_t-(1-K_t)\lambda\bm\nu_t.
\]
For constant mass $m$ and an idealized thrust held for duration $\Delta t$, the feasible
velocity-increment set is $(\Delta t/m)\mathcal U_a$.
With a simultaneous predictable additive velocity increment $\mathbf b_t$, the net applied
increment is $\mathbf b_t-\mathbf a_t$. Since $\mathcal U_a$ is symmetric, feasibility
requires $\mathbf a_t-\mathbf b_t\in(\Delta t/m)\mathcal U_a$.
Provided $\mathbf a_t-\mathbf b_t\in\Pi$, a sufficient condition in every direction within that plane is
\[
[\lambda+K_t(1-\lambda)]\|\bm\zeta_t\|
+(1-K_t)\lambda\|\bm\nu_t\|+\|\mathbf b_t\|
\le(\Delta t/m)\varrho(u_0).
\]
Other command durations or varying mass require the appropriate integrated-force map.
The constraint concerns realized commands. Moments under saturation or truncation follow
from the resulting joint law; observing one feasible command does not establish the unclipped
moment identities.

Rotation-mediated out-of-plane control extends the static geometry to three dimensions
and requires additional attitude dynamics. Matrix gains require corresponding covariance conditions. For a matrix gain, covariance attenuation requires
$(I-K)\Sigma(I-K)'\preceq\Sigma$ (or the corresponding trace inequality for a trace claim).
A scalar predictable gain still requires current sensing to implement its command.
Astrometric measurement error belongs to the separate observer model.

\subsection{Operational implications}

Feasible cancellation constrains what the controller can remove, not which trajectory it intends
to follow afterward. Thrust bounds limit command magnitude rather than its rate of change; the
directional boundary gives the maximum correction, while the inradius guarantees feasibility
in every in-plane direction. Even within those bounds, a controller may adopt the displaced
state and reject subsequent disturbances around its new continuation. Remark~\ref{rem:adoption}
explains why absence of return to the old reference can coexist with a variance response.

\Needspace{5\baselineskip}
\begin{remark}[Reference adoption and subsequent variance suppression]
\label{rem:adoption}
An additional policy could adopt the post-impulse state as a new reference and subsequently
follow its uncorrected continuation. This is a possible extension, not the long-horizon conclusion
of Assumption~\ref{ass:A3}(b), which specifies only onset retention and a local no-steering window.
If the same-history fixed-covariance gain law is also retained, its variance response is still
that of Proposition~\ref{prop:o4}(ii). Neither the absence of visible return nor onset damping
alone identifies this policy.

A deterministic translation leaves the covariance of a given random vector unchanged.
Changing or estimating a reference can change the observer and selected residual law, so a
refitted trajectory does not inherit the variance identity automatically.
Relating onset retention quantitatively to later variance suppression requires a specified
policy, mass, command duration, input information, and resource model.
The inradius supplies sufficient feasibility for commands within the actuator plane,
not the maximum impulse a controller can remove. Interpretation remains conditional on the alternative models and
observation checks.
\end{remark}

A body-generated pulse can leave the same initial velocity change without any controller.
If activity continues, its subsequent forcing can also change the variance response. Thus the
inert--stabilized impulse comparison must be accompanied by a restriction separating body
activity before it supports three-regime identification. Remark~\ref{rem:cometary} identifies
which properties remain shared with the impulsive class and which require its own input law.

\Needspace{5\baselineskip}
\begin{remark}[The impulsively forced case]
\label{rem:cometary}
An impulsively forced model is its own twin. Its positive event variance increment
is imposed in Proposition~\ref{prop:o2o3}, not inferred from the word impulsive.
Component cumulants and variance shares determine how its kurtosis changes with pulse amplitude. Its forecast residual can be stationary, and its isolated
uncorrected velocity impulse can have the same local ramp as an external impulse.
With continuing body inputs, the later response depends on their law and is not covered by
the isolated-impulse Keplerian corollary. No universal C-versus-S separation follows from a
particular response shape or tail description.
\end{remark}

An actuator capable of a command need not be using it. The reachable-set calculation tests
feasibility; the disturbance response tests what the object actually does. A controller can
accept a shifted course (Remark~\ref{rem:adoption}), while continued body activity can change
the motion without control (Remark~\ref{rem:cometary}). Feasibility narrows the possible commands;
Corollary~\ref{cor:joint} identifies a regime when the joint responses separate the admitted
alternatives.

\clearpage
\section{Proofs and supporting derivations}
\label{app:framework}

\subsection{Proof of Lemma~\ref{lem:zeroresidual}: common zero-residual limit}
\label{proof:zeroresidual}

\begin{proof}
The first two conclusions follow by substituting into equation~\eqref{eq:resid}.
The bounded observation maps give
\[
\|P\|:=\max_{t\in\mathcal I}\|P_t\|_{\mathrm{op}}<\infty,
\]
and hence
\[
\|\mathbf y^{(n)}-\mathbf n^{\mathrm{obs}}\|_{\mathcal I}
\le\|P\|\,\|\mathbf e^{(n)}\|_{\mathcal I}\to0.
\]
At either exact endpoint $\mathbf y=\mathbf n^{\mathrm{obs}}$, which proves equality of the observation laws.
\end{proof}

\subsection{Proof of Proposition~\ref{prop:o1}: lag-one projection}
\label{proof:o1}

\begin{proof}
\emph{Projection and covariance identity.} Use the centered process specified in Assumption~\ref{ass:A1},
and put $Q=\E[\mathbf u_t\mathbf u_t']$ for
$\mathbf u_t=\mathbf q_t-\Phi\mathbf q_{t-1}$. The normal equations give
\[
\E[\mathbf u_t\mathbf q_{t-1}']=\Gamma(1)-\Phi\Gamma(0)=0.
\]
Expanding the covariance of $\Phi\mathbf q_{t-1}+\mathbf u_t$ therefore gives
\begin{equation}
\Gamma(0)=\Phi\Gamma(0)\Phi'+Q,\qquad Q\succeq0.
\label{eq:o1lyap}
\end{equation}

\emph{Exclude eigenvalues outside the unit disk.} For any eigenpair
$\Phi'x=z x$, $x\in\mathbb C^d\setminus\{0\}$, apply the Hermitian form to obtain
\[
(1-|z|^2)x^*\Gamma(0)x=x^*Qx\ge0.
\]
Since $\Gamma(0)\succ0$, this implies $|z|\le1$.

\emph{Exclude the boundary using the remote past.} Suppose $|z|=1$. Then
$\E|x^*\mathbf u_t|^2=0$, so $y_t=x^*\mathbf q_t$ satisfies
$y_t=\bar z\,y_{t-1}$ almost surely. For every integer $k\ge1$,
$y_t=\bar z^{\,k}y_{t-k}$ belongs to the complexified closed linear past through $t-k$.
Its real and imaginary parts therefore belong to every real past space. Pure
non-determinism makes both zero, contradicting
$\E|y_t|^2=x^*\Gamma(0)x>0$. Hence $|z|<1$ for every eigenvalue.
The covariance identity alone has a possibly singular decrement $Q$; strictness follows
from this additional remote-past argument, not from a positive-definite Lyapunov decrement.

\emph{Apply the hypotheses to each regime.} Assumption~\ref{ass:A1} supplies these
properties for (I), and Assumption~\ref{ass:A3}(c) for (S). Thus both satisfy $\rho(\Phi)<1$.
\end{proof}

\subsection{Proof of Lemma~\ref{lem:kurtshare}: kurtosis share identity}
\label{proof:kurtshare}

\begin{proof}
\emph{Variance.} Independence and centering give
$s^2:=\operatorname{Var}(R)=(1-\lambda)^2\sigma_Z^2+\lambda^2\sigma_N^2>0$.

\emph{Fourth cumulant.} Expanding the fourth power, the odd cross terms vanish and the even
cross term is $6(1-\lambda)^2\lambda^2\sigma_Z^2\sigma_N^2$. Subtracting $3s^4$ therefore gives
\[
\E R^4-3s^4=(1-\lambda)^4\kappa_Z\sigma_Z^4+
\lambda^4\kappa_N\sigma_N^4.
\]
Division by $s^4$ yields the stated identity. The same calculation holds in each conditional
law satisfying the declared assumptions.
\end{proof}

\subsection{Proof of Proposition~\ref{prop:o2o3}: disturbance-conditioned moment comparisons}
\label{proof:o2o3}

\begin{proof}
\emph{Uncorrected variances.} With correction absent, the residual and its twin coincide.
Thus their conditional variances and every defined shape statistic coincide. The stipulated
positive event increment gives part~(i); in (C), conditional independence adds component
covariances, giving the assumed positive sum of increments.

\emph{Uncorrected shape.} Apply Lemma~\ref{lem:kurtshare} separately in each admissible
conditioning set. Subtracting the quiet weighted sum from the event weighted sum proves the
if-and-only-if condition in part~(ii). Increasing a pulse amplitude changes its variance share; the weighted cumulants determine the
resulting excess kurtosis. Under the separate cross-regime
condition, $\kappa_b\le\kappa_I$, $\kappa_e=\kappa_I\ge0$, and hence
\[
\kappa_C=w_b^2\kappa_b+w_e^2\kappa_I
\le(w_b^2+w_e^2)\kappa_I\le\kappa_I,
\qquad w_b+w_e=1.
\]

\emph{Stabilized variance.} Write $c_t=1-K_t$, $a=(1-\lambda)^2$ and
$b=\lambda^2\operatorname{tr}F_\nu\ge0$. Conditional independence gives
$\operatorname{tr}H_t=c_t^2(aV_t+b)$. Since $c_t=1$ in quiet windows,
\[
\Delta H=a\Delta V-
\E[(1-c_t^2)(aV_t+b)\mid\mathrm{event}]\le a\Delta V.
\]
The subtracted integrand is nonnegative. If $\lambda>0$, then $a<1$ and
$\Delta V>0$ gives strict attenuation. At $\lambda=0$, strictness is exactly the
positive-expectation condition in the statement. When $c_t=c$ is constant in both windows,
the common addend cancels and $\Delta H=c^2a\Delta V$, as asserted.
When the sensing floor obeys the trace bound of Assumption~\ref{ass:A3}(d)---for $\lambda>0$,
$\lambda\operatorname{tr}F_\nu<(2-\lambda)V_t$---we have
\[
aV_t+b<V_t\iff\lambda\operatorname{tr}F_\nu<(2-\lambda)V_t,
\]
and multiplication by $c_t^2\le1$ preserves this level inequality, so $\operatorname{tr}H_t<V_t$;
this is the level ordering that the cross-sectional comparison against a validated ballistic cohort
uses (Appendix~\ref{app:mc}).

\emph{Stabilized conditional shape.} In the scalar law of part~(iii), the lemma gives
$\kappa_R=w^2\kappa_Z$. Positive variances and $0<\lambda<1$ imply $0<w<1$;
with $\kappa_Z>0$ this proves both strict inequalities. At the two endpoints $R$ is,
respectively, the disturbance or the floor. Given the history, multiplication by the known
positive $c_t$ cancels from standardized central moments. No step replaces these conditional
laws by a pooled event law without the additional assumptions stated in the proposition.
\end{proof}

\subsection{Proof of Proposition~\ref{prop:o3inv}: conditional shape invariance and pooling}
\label{proof:o3inv}

\begin{proof}
\emph{Conditional scaling.} In any fixed scalar projection, let $s_t^2>0$ be the
conditional variance of $r_t$. Given $\Fcal_{t-1}$, $c_t>0$ is fixed, so
\[
\frac{c_t r_t-\E[c_t r_t\mid\Fcal_{t-1}]}{
\sqrt{\operatorname{Var}(c_t r_t\mid\Fcal_{t-1})}}
=\frac{r_t-\E[r_t\mid\Fcal_{t-1}]}{s_t}.
\]
Every defined standardized central moment is therefore unchanged. Adding a predictable term
also disappears on centering.

\emph{Pooling scales.} For the separately specified independent $S,Z$,
$\E X^2=\E S^2$ and $\E X^4=(\kappa_Z+3)\E S^4$.
Division yields the formula in the statement; $\operatorname{Var}(S^2)\ge0$ gives its
inequality. For two mixtures with the same standardized law, the formula compares their
respective fourth-to-second scale-moment ratios.

\emph{Vary the cancellation share with the component laws fixed.} The share identity gives
$\kappa_R=w^2\kappa_Z$ and, for $0<\lambda<1$,
\[
w=\left[1+\frac{\lambda^2}{(1-\lambda)^2}
\frac{\sigma_N^2}{\sigma_Z^2}\right]^{-1}.
\]
The positive ratio $\lambda/(1-\lambda)$ is strictly increasing, so $w$ decreases strictly
from $1$ to $0$. Since $\kappa_Z>0$, the stated strict decrease follows, including its
endpoint limits. With a nonzero floor excess the formula is instead
$w^2\kappa_Z+(1-w)^2\kappa_N$; it exceeds the floor, for $w>0$, exactly when
$w\kappa_Z>(2-w)\kappa_N$. This inequality is not automatic for a near-Gaussian floor.
\end{proof}

\subsection{Proof of Proposition~\ref{prop:o4}: local impulse and variance responses}
\label{proof:o4}

\begin{proof}
\emph{Propagate the retained initial velocity.} For $T=h\Delta t$, the stipulated
transition gives
\[
\begin{pmatrix}1&\Delta t\\0&1\end{pmatrix}^{h}
\begin{pmatrix}0\\\eta\delta\end{pmatrix}
=\begin{pmatrix}\eta\delta h\Delta t\\\eta\delta\end{pmatrix}.
\]
Use $\eta=1$ in (I). The stipulated first moments make both continuation positions
integrable at each finite horizon. Conditional centering and the tower property give zero
future innovation means given $\mathcal H_0$. The equal expected predictable inputs therefore
cancel in the mean contrast; Assumption~\ref{ass:A3}(b) excludes additional impulse-induced
mean steering on this horizon.
Thus the displacement entries give part~(i); an unknown input size absorbs the factor $\eta$.
Under gravity these formulas apply on the local horizon specified in the statement.

\emph{Compute the variance difference.} In (I), Assumption~\ref{ass:A2} fixes
$H=g(W_t)$ in both quiet-environment continuations. The common conditional environment
law and the stipulated finite expected trace give equal finite covariance expectations,
so $v_h^I=0$. In (S), the stipulated
covariance before recovery scaling is the same $\Sigma$ in both continuations, whereas their gains are
$K(h)$ and $0$. Hence
\[
v_h^S=\operatorname{tr}[(1-K(h))^2\Sigma]-\operatorname{tr}\Sigma
=-K(h)(2-K(h))\operatorname{tr}\Sigma.
\]

\emph{Check signs and scope.} Since $\operatorname{tr}\Sigma>0$ and $0\le K<1$,
the expression is zero at $K=0$ and negative at $K>0$. Its derivative in $K$ is
$-2(1-K)\operatorname{tr}\Sigma<0$, so the nondecreasing gain path makes the response
nonincreasing. Continuity gives the stated limit for the idealized continuation. Both
conditional covariances remain positive definite: negativity concerns their difference in
trace only. If the two gains instead equal the same constant, their conditional covariances
coincide and the response is zero.
\end{proof}

\subsection{Proof of Corollary~\ref{cor:joint}: identifiable uniqueness}
\label{proof:joint}

\begin{proof}
Each uncorrected regime is its own unstabilized twin, so its defined variance and
shape contrasts are zero. Let $w_b\in(0,1)$ denote the body channel's share of
conditional variance in (C). Lemma~\ref{lem:kurtshare} gives
\[
\begin{split}
\kappa_I-\kappa_C
&=2w_b(1-w_b)\kappa_I+w_b^2(\kappa_I-\kappa_b)\\
&\ge2w_b(1-w_b)\kappa_I>0.
\end{split}
\]
Thus (C) fails all three alternatives in
equation~\eqref{eq:endpoint-reduction}, even if it satisfies O1. The inert model
satisfies O1 and the equality $\kappa_R=\kappa_I$. An (S) model with positive
$d_2$ or $d_3$ satisfies the corresponding alternative. At $\lambda=0$,
Proposition~\ref{prop:o3inv} gives
$\kappa_S=\kappa_S^{\mathrm{pass}}=\kappa_I$ under the stated matching condition,
so that stabilized model satisfies the equality alternative as well.

The strict O2 inequality is Proposition~\ref{prop:o2o3}(iii). The O3 identity
follows from Lemma~\ref{lem:kurtshare}: the independent zero-excess floor
contributes variance but no fourth cumulant, giving $0\le w<1$ and the displayed
strict difference. Positive recovery scaling leaves standardized conditional
moments unchanged. Proposition~\ref{prop:o4} gives the O4a ramps and the O4b
variance contrast, using only the fixed covariance before recovery scaling and gain schedule
for the latter. For constant gain, direct subtraction gives
\[
L_I-L_S
=K(2-K)V+(1-K)^2\lambda\bigl[(2-\lambda)V-\lambda F\bigr],
\]
which is strictly positive under either stated level-comparison condition.
The first stage excludes (C); a strict suppression contrast then excludes (I)
from the retained pair. On the stipulated reference family, (C) fails the first
stage, (I) passes it with zero suppression, and (S) passes it with strict
suppression. These joint outcomes are distinct.

For part~(iii), the detectable-residual condition restricts the admitted family away from the
common observation limit of Lemma~\ref{lem:zeroresidual}. The distance condition is a separate
premise and supplies the bound used here. Since the true signature belongs to $\mathcal S_R$
and its estimation error is at most $\varepsilon$, the true regime is compatible.
For any rival signature $u$, the triangle inequality gives
\[
\|\widehat s-u\|_a\ge\|s-u\|_a-\|\widehat s-s\|_a
\ge\delta_R(s)-\varepsilon>\varepsilon.
\]
Taking the infimum over each rival set preserves the strict lower bound
$\delta_R(s)-\varepsilon>\varepsilon$, excluding every rival regime.
The probabilistic statement follows by applying this argument on the stated error-bound event.
\end{proof}

\subsection{Proof of Corollary~\ref{cor:classconsistency}: class consistency}
\label{proof:classconsistency}

\begin{proof}
Set $Q_T(R)=-D_R(\widehat s_T)$ and $Q(R)=-D_R(s)$. Distance to a nonempty set is
1-Lipschitz, so
\[
\max_{R\in\mathcal R}|Q_T(R)-Q(R)|\le\|\widehat s_T-s\|_a\longrightarrow0
\]
in probability. This gives the uniform criterion convergence in the extremum argument.
If $\mathcal R(s)=\{I,C,S\}$, the membership statement holds for every $T$.
Otherwise, $Q(R)=0$ on $\mathcal R(s)$ and $Q(R)\le-\gamma(s)$ outside it.
When the uniform criterion error is less than $\gamma(s)/2$, every compatible class
has a higher sample criterion than every excluded class. Every maximizer therefore belongs
to $\mathcal R(s)$, giving the probability result. Under almost-sure convergence,
this ordering holds for all sufficiently large $T$ almost surely.
The same argument applies to a consistently estimated enlarged signature with a
positive gap from every rival regime.
\end{proof}

\subsection{Proof of Corollary~\ref{cor:classnull}: orbital impulse responses}
\label{proof:classnull}

Reference convention, orbital energy, initial conditions, response horizon, and observational
uncertainty all matter. A single apparition can already contain appreciable tidal curvature;
multiple apparitions are not a mathematical requirement for evaluating a propagated uncorrected null.
Refitting to post-impulse data can absorb some displacement into estimated elements, depending
on sampling and fit design, so the reference used for the response must be declared.

For an eccentric bound reference, the first-order radial secular envelope has rate
$e/(1+e)$ times the transverse envelope rate when $\delta a\ne0$; the proof uses the
standard linearized relative motion about an elliptic reference orbit
\citep{tschauner1967elliptic,yamanaka2002elliptic,carter1990keplerian}.
These maxima occur at different phases, so their ratio is not a pointwise component ratio.
The linearized rates do not imply unbounded exact separation of two bound ellipses.
For hyperbolic relative motion, the leading linear term is the difference of outgoing velocities.

\begin{proof}
\emph{Free flight and the local remainder.} The double-integrator identity gives
$\Delta\mathbf r(T)=T\bm\delta$. In the first-order gravitational variational equation,
$\ddot{\mathbf d}=A(t)\mathbf d$, $\mathbf d(0)=0$, $\dot{\mathbf d}(0)=\bm\delta$,
put $L=\sup_{[0,T]}\|A(t)\|$. The integral equation is
\[
\mathbf d(t)=t\bm\delta+\int_0^t(t-s)A(s)\mathbf d(s)\,ds.
\]
Picard iteration of its nonnegative norm bound gives
\[
\|\mathbf d(t)\|\le\|\bm\delta\|\sinh(\sqrt L t)/\sqrt L
\]
(with the continuous $L=0$ value). Substituting this bound back into the integral equation gives the
local remainder
\[
\|\mathbf d(t)-t\bm\delta\|\le\|\bm\delta\|\bigl[\sinh(\sqrt L t)/\sqrt L-t\bigr]
=O(\|\bm\delta\|Lt^3)\quad\text{as }Lt^2\to0.
\]
For a Kepler reference, $\|A(t)\|=2\mu/r(t)^3$, so $LT^2=2\chi^2$ on that window.
The nonlinear approximation additionally requires small relative displacement, as stipulated.

\emph{Bound-orbit variation --- secular rate.} At fixed initial radius, differentiating
$-\mu/(2a)=\|\mathbf v\|^2/2-\mu/r$ gives the stated $\delta a$;
differentiating $n=\sqrt{\mu/a^3}$ gives $\delta n=-3n\delta a/(2a)$.
All Kepler elements except mean anomaly are constant, so the only secular first-order term is the
anomaly difference $\delta n(t-t_0)$; the remaining first-order position contributions are periodic.

\emph{Envelope maxima.} For eccentric anomaly $E$ and $r=a(1-e\cos E)$, differentiation at fixed
elements gives
\[
\frac{\partial r}{\partial M}=\frac{ae\sin E}{1-e\cos E},\qquad
r\frac{\partial\theta}{\partial M}=
\frac{a\sqrt{1-e^2}}{1-e\cos E}.
\]
The maximum absolute radial coefficient is $ae/\sqrt{1-e^2}$, attained at $\cos E=e$, and the maximum
transverse coefficient is $a\sqrt{(1+e)/(1-e)}$, attained at periapsis. Multiplying these by the
secular rate $|\delta n|=\tfrac32 n|\delta a|/a$ gives the envelope rates
\[
R_R=\tfrac32 n|\delta a|\,\frac{e}{\sqrt{1-e^2}},\qquad
R_T=\tfrac32 n|\delta a|\,\sqrt{\frac{1+e}{1-e}},\qquad
R_R/R_T=\frac{e}{1+e}.
\]

\emph{Bounded and degenerate cases.} The transverse coefficient never vanishes, so a nonzero
$\delta a$ precludes bounded in-plane variation; if $\delta a=0$, every remaining first-order term is
periodic. Out-of-plane motion changes only the fixed orientation of the orbital plane at first order,
and its projection is a linear combination of the bounded periodic reference coordinates.
The circular case follows by setting $e=0$, or directly from its nonsingular Cartesian
variation \citep{clohessy1960terminal}.

For two exact bound ellipses, independently of the linearized calculation,
\[
\|\mathbf r_1(t)-\mathbf r_0(t)\|
\le a_1(1+e_1)+a_0(1+e_0)<\infty.
\]
Thus secular linearized phase error is not exact unbounded separation. The exact energy
change is $\Delta\mathcal E=\mathbf v\cdot\bm\delta+\|\bm\delta\|^2/2$.

\emph{Hyperbolic leading term.} For an escaping positive-energy conic,
$\dot r^2=2\mathcal E+2\mu/r-h_{\rm ang}^2/r^2$ on its outgoing branch, so
$r(t)/t\to\sqrt{2\mathcal E}>0$ and the vector acceleration is $O(t^{-2})$.
Its integrability gives an outgoing velocity $\dot{\mathbf r}(t)=\mathbf v_\infty+O(t^{-1})$, and a
further integration yields the leading position expansion
\[
\mathbf r(t)=\mathbf v_\infty t+O(\log t).
\]
Subtracting the two expansions, if $\Delta\mathbf v_\infty\ne0$ then
$\|\Delta\mathbf r(t)\|/t\to\|\Delta\mathbf v_\infty\|>0$.

\emph{Parabolic boundary.} For $p>0$, orthonormal periapsis and transverse directions
$\mathbf P,\mathbf Q$, and $D=\tan(\nu/2)$, the outgoing parabolic orbit satisfies
\[
\mathbf r=\tfrac p2(1-D^2)\mathbf P+pD\mathbf Q,\qquad
t-t_p=\frac{p^{3/2}}{2\sqrt\mu}(D+D^3/3).
\]
Writing $x=(6\sqrt\mu(t-t_p)/p^{3/2})^{1/3}$, inversion gives
$D=x-x^{-1}+O(x^{-3})$. Substitution, with fixed $t_p$, yields
\[
\mathbf r(t)=-(9\mu/2)^{1/3}t^{2/3}\mathbf P+
(6\sqrt\mu)^{1/3}\sqrt p\,t^{1/3}\mathbf Q+O(1).
\]
Subtract these expansions when both trajectories remain parabolic; leading coefficients can
cancel, so the rate depends on their directions and $p$. If the perturbed orbit becomes bound,
its position stays bounded while the unchanged parabolic reference has norm asymptotic to
$(9\mu/2)^{1/3}t^{2/3}$, proving the stated relative growth.
Finally, every bound or parabolic trajectory is $o(t)$, whereas a hyperbolic trajectory divided
by $t$ tends to its nonzero outgoing velocity. This proves the stated mixed-conic linear cases.
Rebinding one trajectory does not make the comparison bound--bound.
\end{proof}

\subsection{Stability, cadence, and memory}
\label{app:stability}

A disturbance can cease affecting future inputs while its accumulated displacement remains.
Likewise, a stationary input law can permit rare kicks large enough to violate a physical
tolerance. The conditions used for moment comparisons, forgetting of initialization, and
trajectory control therefore answer different questions. They must remain separate when richer
disturbance models are used to construct a reference signature.

\paragraph{Stability distinctions.} The covariance recursion describes the variability of new
velocity inputs, whereas a forcing recursion describes how preceding inputs affect later ones.
Their stability conditions serve these different roles \citep{andree2026stability}. For constant BEKK coefficients the spectral
condition $\rho(\Acal\otimes\Acal+\Bcal\otimes\Bcal)<1$ is a second-moment condition for
the innovation, with stationary integrable exogenous loading and the appropriate innovation
assumptions. It concerns $\E\operatorname{tr}H_t<\infty$, not automatically
$\E\|H_t\|^2<\infty$. Under the inert restriction $H_t=C+G(W_t)$, stationarity and
integrability still require those properties of $G(W_t)$.
A negative top Lyapunov exponent, together with the integrability hypotheses of a stochastic
recurrence theorem, can establish strict stationarity and forgetting without finite second
moments \citep{bougerol1992strict,straumann2006quasi}. It is not a replacement for a moment
condition. Likewise a constant block $\rho(\Mcal)<1$ gives geometric decay of differences
from initialization, while a stationary finite-variance driven solution also needs a stationary
finite-variance driver. Proposition~\ref{prop:o1} uses the explicit forecast-residual assumptions
directly, through Assumptions~\ref{ass:A1} and~\ref{ass:A3}(c).

Temporary amplification is compatible with contraction over a sequence of updates. For a scalar
linear recurrence, successive multipliers combine by multiplication: an expansion by a factor
of two followed by contraction by a factor of one quarter leaves half the initial difference.
The corresponding log-average condition concerns the product's long-run rate. Under the
stationarity, ergodicity, and logarithmic integrability conditions of a stochastic recurrence
theorem, contraction on average can yield a unique stationary solution and forgetting of
initialization even when some updates expand \citep{straumann2006quasi}. For matrix dynamics,
the accumulated matrix products or suitable Lipschitz bounds govern this conclusion; individual
spectral radii do not give the general criterion.

A planner also needs to know how large the deviation can become before later correction reduces
it. Even the scalar example doubles the initial difference before halving it. Bounding the
duration of an expanding episode does not by itself bound its amplification or the size of a
new disturbance. A hard tracking guarantee requires bounds on the intervening dynamics and
admissible inputs; with stochastic inputs, a specified finite-horizon exceedance probability
can instead be assessed. Remark~\ref{rem:o1phys} states these physical criteria. O1 remains a
property of the analyst's forecast residual, not a claim that every control update contracts or
that a long-run stability condition enforces the planner's tolerance.

Sampling also changes which physical actions are resolved. Several disturbances and corrections
can occur between two observations, leaving only their accumulated effect in the measured record.
A regular residual at that cadence need not describe the intermediate dynamics. Remark~\ref{rem:frequency}
keeps the statistical property attached to the cadence and transformation at which it is defined,
rather than interpreting it as a contraction law at every physical time scale.

\Needspace{5\baselineskip}
\begin{remark}[Sampling frequency and the operational form of contraction]
\label{rem:frequency}
Statistical properties belong to the stated diagnostic cadence and transformation.
Changing cadence can expose dynamics omitted by the coarser model and change finite-sample
memory estimates. A finite sum of stable short-memory linear components can resemble slow
decay over a finite range, but neither an increased memory estimate nor its cause follows
automatically \citep{granger1980aggregation,shimotsu2005whittle}.

A negative top Lyapunov exponent with the required recurrence and integrability hypotheses
can establish forgetting of initial conditions \citep{bougerol1992strict}; a narrow observed
envelope does not establish that exponent. Covariance stationarity, strict stationarity,
forgetting, and the physical tolerance of Remark~\ref{rem:o1phys} remain distinct.
The population projection is not a per-step dynamic contraction law.
\end{remark}

Slowly decaying motion can reflect persistent forcing or gradual adjustment by a controller.
A memory estimate by itself does not choose between those causes. To add it to a classification
signature, the candidate forcing and steering models must imply different observable persistence
after the same observer and sampling procedure are applied. Remark~\ref{rem:memory} identifies
the additional restrictions needed before that extension can supply class separation.

\Needspace{5\baselineskip}
\begin{remark}[Possible memory extensions]
\label{rem:memory}
A reproducible slow component across matched windows could motivate richer uncorrected and steering
reference models. Assumptions~\ref{ass:A1}--\ref{ass:A2} allow stationary long-memory residuals
and persistent environments, while Assumption~\ref{ass:A3} leaves long-horizon steering
unspecified. A further diagnostic would therefore require restrictions on drivers, observer,
steering, cadence, and estimation window, followed by its own separation argument.
Finite near-canceling roots can imitate slow decay in finite samples; fractional models must
respect their stationary range. Estimated memory alone supplies no class exclusion here.
\end{remark}

\subsection{Forcing composition and temporal extensions}

Cometary volatile release motivates an outgassing contribution \citep{mumma2011composition}.
The fitted Marsden law describes its smooth orbit-scale mean recoil, while the pulsed channel
represents departures from that mean. The diagnostic model treats these departures as discrete
velocity inputs. Event amplification and fourth-moment properties are imposed where a result
uses them. Body-generated forcing also includes smooth secular effects: Yarkovsky-type thermal-recoil acceleration, routinely detected astrometrically on sub-kilometer
bodies \citep{chesley2003yarkovsky,farnocchia2013yarkovsky,vokrouhlicky2015yarkovsky}, and the spin-averaged
part of shape-modulated radiation pressure. This component is smooth and slowly varying at the diagnostic
scale, so the analysis groups it with $\fext_t$, and $\fbody_t$ denotes the impulsive body component throughout. Spin modulation resolved at the diagnostic scale requires a model beyond this
smooth-mean approximation and is considered with the other alternative mechanisms
(Appendix~\ref{app:empirical}). In the residual identity, $\mathbf s_t$, $\mathbf o_t$,
$\fext_t$, and $\fbody_t$ always denote length-valued propagated contributions.
The symbols $\bm\varepsilon_t$, $\bm\eta_t$, $\bm\zeta_t$, and $\bm\nu_t$ denote
velocity-increment innovations or channels; they are not identical to those displacements.

At coarser cadence the input $\mathbf w_t$ may aggregate lagged innovations; it is generally
not itself a one-step innovation. The BEKK-type covariance recursion
\citep{engle1982arch,bollerslev1986garch,engle1995bekk} applies to $\bm\eta_t$.
The spectral second-moment condition, strict-stationarity conditions, and any fourth-moment
conditions require separate checks
\citep{bougerol1992strict,straumann2006quasi}.

A forcing episode or a corrective response need not finish within one sampling interval.
Its remaining effect can depend on recent inputs, including inputs in other directions. A
finite-lag linear recursion represents this dependence by adding a new disturbance to the
remaining contributions of previous steps. Write $\mathbf f_t$ for deviations of velocity-input
components from their means, with
$\mathbf f_t=\mathbf A_1\mathbf f_{t-1}+\cdots+\mathbf A_p\mathbf f_{t-p}+\mathbf u_t$,
where $p$ is a positive integer and the driver is stationary with finite variance. The coefficient
matrices describe how earlier components contribute to the present input; their off-diagonal
entries allow directional coupling. Stacking the recent inputs expresses the same recursion
through the companion matrix
\[
\Mcal=\begin{pmatrix}\mathbf{A}_1&\cdots&\mathbf{A}_{p-1}&\mathbf{A}_p\\
\mathbf{I}&&&\mathbf{0}\\
&\ddots&&\vdots\\
&&\mathbf{I}&\mathbf{0}\end{pmatrix},
\]
with $\rho(\Mcal)<1$ giving a convergent stationary linear solution under these driver
assumptions \citep{hamilton1994,lutkepohl2005}. The condition makes the contribution of the
initial input history decay; fresh disturbances continue to drive the process. It concerns
persistence of the inputs, not return of accumulated displacement. For identification, such a
recursion supplies a richer forcing or correction history that must still be propagated through
position dynamics and the observer before its signature is compared.

A longer-horizon correction policy may act on deviations of velocity components from a
specified relationship, rather than merely rescale each new disturbance. A velocity-error-correction
model represents this possibility: selected combinations of residual velocities measure the
imbalance, and a response matrix determines how that imbalance changes later velocity. Separating
this response from adjustment to the preceding velocity change gives the specification
\[
\Delta\dot{\mathbf e}_t=
\alpha\bm\beta'\dot{\mathbf e}_{t-1}
+\Gamma\Delta\dot{\mathbf e}_{t-1}+\bm\varepsilon_t.
\]
Here $\bm\beta'\dot{\mathbf e}_{t-1}$ measures the specified velocity imbalance,
$\alpha$ maps it into a corrective change, and $\Gamma$ represents dependence on the preceding
change. For that response to be restoring, its signs and dynamics must satisfy the rank, root,
and driver conditions of a specified cointegrated model
\citep{engle1987cointegration,johansen1991vecm}. Adjustment of a velocity relationship also
leaves the target displacement unspecified: the body may settle into motion about a shifted
path rather than return to the old one. The local free-flight input model sets both adjustment
terms to zero; setting $\alpha=0$ alone retains the dynamics governed by $\Gamma$.
Integration orders likewise depend on the driver and roots. An identification extension using
this policy would compare its propagated velocity and displacement responses with those of
an explicitly matched uncorrected model.

A finite stable linear system driven by short-memory input can yield a short-memory scalar
marginal, with near-canceling roots resembling slower decay in finite samples.
Long memory can result from aggregation over a continuum of suitable adjustment speeds
\citep{granger1980aggregation} or from long-memory input. A finite number of actuators alone
implies neither mechanism. Using memory for regime identification requires separate restrictions
of the kind stated in Remark~\ref{rem:memory}.

\subsection{Benchmark portability and object coupling}
\label{app:portability}

Two uncorrected bodies need not fluctuate equally under a common environment. Different coupling,
for example through area-to-mass ratio, can change the amplitude of their response. A quieter
candidate therefore need not be corrected, and a more variable reference need not be its valid
twin. Cross-object comparison requires a way to account for those physical differences as well
as the observation process.

One tractable case assumes that coupling rescales the entire environmental variance profile by
a constant positive factor. An event-to-quiet ratio then compares each body's relative response
rather than its absolute amplitude. Suppose $V_t=c\,\gamma(W_t)$ with $c>0$ and the same
event/quiet laws of $\gamma$.
Write $V_e=\E[V_t\mid\mathrm{event}]>V_q=\E[V_t\mid\mathrm{quiet}]>0$.
Then $V_e/V_q$ cancels $c$. In (S), set $a=(1-\lambda)^2$ and
$b=\lambda^2\operatorname{tr}F_\nu\ge0$.
For quiet gain zero, event gain in $[0,1)$, and $aV_q+b>0$,
\[
\frac{\E[\operatorname{tr}H_t\mid\mathrm{event}]}{aV_q+b}
\le\frac{aV_e+b}{aV_q+b}\le\frac{V_e}{V_q}.
\]
The first inequality follows by dropping event gain attenuation, and is strict exactly when
the removed event variance has positive expectation. The second follows on cross-multiplication
from $b(V_e-V_q)\ge0$ and is strict when $b>0$.
With the same constant gain in both windows, that gain cancels from the ratio.
In particular, when $b=0$ the ratio equals the uncorrected ratio despite attenuation of levels.
Thus ratio strictness is not identical to increment strictness.

For the general additive inert covariance $C+G(W_t)$, this portability argument applies when
the object-specific baseline and variable parts share a common multiplicative factor.
Computing the ratio requires positive quiet variance; applying it to measurements also requires
matching the observation operators and measurement floors.

\clearpage
\section{Empirical implementation considerations}
\label{app:empirical}

A smaller observed variance can reflect rejected disturbances, a quieter environment, weaker
physical coupling, or a change in how the trajectory was measured. Applying
Algorithm~\ref{alg:classification} requires the physical contrast to remain identifiable after
these alternatives are admitted to the observation model. Table~\ref{tab:mimics} records
examples and discriminating checks; the list is not exhaustive. The estimation and calibration
requirements below concern how each comparison is recovered from an actual trajectory record.

\begingroup
\small\singlespacing
\setlength{\tabcolsep}{4pt}
\setlength{\LTpre}{6pt}
\setlength{\LTpost}{6pt}
\begin{longtable}{@{}>{\raggedright\arraybackslash}p{.22\linewidth}>{\raggedright\arraybackslash}p{.18\linewidth}>{\raggedright\arraybackslash}p{.18\linewidth}>{\raggedright\arraybackslash}p{\dimexpr.42\linewidth-6\tabcolsep\relax}@{}}
\caption{\textbf{Alternative mechanisms and observational effects.} Processes that can resemble stabilized signatures, with checks for distinguishing them.}\label{tab:mimics}\\
\toprule
Mechanism & Relation to maintained assumptions & Possible effect & Discriminating check \\
\midrule
\endfirsthead
\multicolumn{4}{@{}l}{\textbf{Table~\thetable{} (continued).}}\\
\toprule
Mechanism & Relation to maintained assumptions & Possible effect & Discriminating check \\
\midrule
\endhead
\midrule
\multicolumn{4}{r@{}}{\footnotesize Continued on the next page.}\\
\endfoot
\bottomrule
\endlastfoot
Spin damping with spin-dependent recoil & may breach A2 through spin--recoil coupling & lower variance if recoil fluctuations decline & model spin, recoil, and observation jointly; compare the predicted response and damping timescale \citep{burns1973nutation,pravec2005tumbling,vokrouhlicky2015yarkovsky} \\
Post-outburst activity cessation & regime (C), not (I) & $\hat v_h<0$ & coincident shift in the level channel / fitted non-gravitational parameters \\
Post-event follow-up upgrades (denser, better astrometry) & observing process & $\hat v_h<0$ & recompute on a homogeneous observing subset spanning $t_0$; report observing metadata \\
Outlier rejection and robust weighting in the reduction & observing process & $\Delta\hat H$ and $\Delta\hat\kappa$ attenuated & rerun on a fully weighted fit without outlier rejection \\
Measurement-share shifts across windows & observing process & changes in variance or pooled shape & forward-model measurement and filtering jointly; compare scalar moments under the fitted observation law \\
Reference-fit absorption: whole-arc non-gravitational parameters absorb event-elevated forcing & estimator & $\Delta\hat H$ attenuated & hold the event window out of the reference fit; simulate the projection null on an uncorrected body \\
Quieting environment over the response window & breaches fixed-covariance comparison & potentially negative $\hat v_h$ & model covariance before recovery scaling and after observation, using the appropriate law at each stage \\
Own-shock asymmetric feedback & breaches A2 & potentially negative $\hat v_h$ & fit a model allowing shocks to lower variance and determine its sign under the fitted parameters \\
Many-jet aggregation & possible regime (C) & lower excess kurtosis can accompany higher variance & evaluate component cumulants and variance shares; amplitude increases alone do not fix the kurtosis increment \\
Photocenter displacement: a sub-threshold coma shifts the fitted position, co-moving with $W_t$ & observing process (bias, not variance) & event-window bias in $\mathbf{e}_t$ imitating variance attenuation or excess-kurtosis reduction & wavelength/aperture dependence of the fitted position; stacked-profile asymmetry; deweight per cometary practice \citep{veres2017statistical} \\
Execution noise on the recovery channel & loop-side law beyond A3's scaling alone & independent zero-excess noise can reduce positive excess kurtosis & use a validated execution-noise model; scaling with command magnitude is an additional premise \\
\end{longtable}
\footnotesize
\noindent\textit{Notes.} A1--A3 refer to Assumptions~\ref{ass:A1}--\ref{ass:A3}.
The effects listed are possible effects, not universal sign laws. The checks can constrain a
specified alternative but are not an exhaustive identification procedure. Reduction and reference
fitting must be simulated as actually used; neither attenuates every statistic in every data set.
\par
\endgroup

\paragraph{Covariance model and environment index.}
The inert benchmark must distinguish changes associated with the external environment from
feedback following the object's own shocks. The uncorrected covariance recursion~\eqref{eq:hrec}
can be fitted by quasi-maximum likelihood
under an identified specification and suitable regularity conditions
\citep{engle1995bekk,straumann2006quasi}. Testing $\Acal=\Bcal=0$ is
nonregular in general. The first derivative in quadratic BEKK coordinates vanishes there;
reparameterizing into covariance coefficients does not by itself supply identification, a
nondegenerate information matrix, or a particular tangent cone. Covariance-persistence nuisance
parameters can also be unidentified under a homoskedastic null.
Consequently neither an ordinary Lagrange multiplier reference distribution \citep{engle1982arch} nor a generic
chi-bar-squared law or conservativeness claim follows from the boundary alone
\citep{andrews2001boundary}. A specific test needs its own null theory or validated simulation.

Constructing $W_t$ from ephemeris and observing metadata can reduce direct outcome feedback.
It does not prove causal exogeneity, independence from reference-estimation errors, or that
the environment is fixed in an intervention comparison.

\paragraph{Benchmark construction.}
A quiet window is meant to represent weaker forcing, rather than observations selected because
the realized trajectory happens to be quiet. Select such windows by a pre-specified environmental
rule and model their
relationship to the same-history counterfactual. Selection on low realized variance can lower a
baseline and push an estimated difference upward. A whole-sample baseline mixes event and quiet
conditions; the resulting bias depends on their levels and weights and has no universal negative sign.
With common event and quiet environmental laws, a constant multiplicative difference in object
coupling cancels from the event-to-quiet variance ratio. Event-dependent correction or an additive
sensing floor can lower that ratio, whereas a common constant gain cancels from it.
Appendix~\ref{app:portability} derives this portability condition; it requires the baseline and
environmental variation to share the same multiplicative coupling.

\paragraph{Event timing.}
The impulse response compares histories that differ by an identified disturbance, so inference
must distinguish that intervention from fluctuations that happen to trigger a detector. A single
apparition may contain only one effectively independent event. Its null response
distribution then requires strong time-series or model assumptions. A pre-specified aggregate
statistic or simultaneous band avoids treating pointwise bands as joint coverage, but must itself
be calibrated. Null-imposed residual resampling is one possible method, not a validity theorem.

An i.i.d.\ residual bootstrap requires conditions making the empirical residual law and the target
functional consistent; mere exchangeability or conditional standardization is insufficient.
Serial dependence may require a justified block, dependent-multiplier, or model-based procedure.
A wild bootstrap result for a unit-root statistic \citep{cavaliere2008bootstrap} does not automatically
validate a conditional-variance response statistic. Fourth-moment and dependence conditions matter.
Subsampling or $m$-out-of-$n$ methods require their own rate and tuning arguments and are not
universally conservative defaults.

An observed event time selected by a displacement detector is not automatically an exogenous
intervention, even when $W_t$ is flat. A detector may lag the start of a burst and select its quieting
aftermath; sensitivity to timing and the baseline must therefore be assessed.
If selection is part of the statistic, simulate that selection inside the null procedure.
Reapplying a detector is necessary to represent that selection mechanism in such a simulation,
but is not sufficient for general bootstrap validity. Appendix~\ref{app:mc} reports performance
for particular detector and forcing laws.

\paragraph{Variance-response inference.}
The diagnostics may be applied in measured coordinates $\mathbf y_t$ or to a specified
transformation $u_t=\mathcal L(\mathbf y)_t$. Its output requires its own statistical assumptions.
The notation $L(B)$ is reserved for linear filters; moving-median residuals are nonlinear.

The physical inputs are integrated into position and then filtered by the instrument and observer.
Environment adjustment must therefore use the covariance of the quantity actually observed.
Dividing by the inert input covariance $g(W_t)$ does not generally fix the stabilized covariance
before recovery scaling, which also contains an additive floor, or the covariance after state propagation and filtering.
The gradient experiments illustrate a specified scalar covariance model, including an estimated
version within its assumed family.

For a fixed scalar projection $\ell$, conditional shape uses
\[
z_{\ell,t}=
\frac{\ell'\bm\varepsilon_t-\E[\ell'\bm\varepsilon_t\mid\Fcal_{t-1}]}
{\sqrt{\ell'H_t\ell}},
\qquad \ell'H_t\ell>0.
\]
An observed forecast residual needs its own conditional mean and covariance in this formula.
Vector whitening uses $H_t^{-1/2}\bm\varepsilon_t$ on a positive-definite subspace; the scalar
kurtosis identity does not extend to its norm. Conditional standardization does not establish
independence. If conditional means are zero and variances one, a pooled fourth moment is the
average conditional fourth moment when integrable; the conditional shapes can still vary.
Window-constant scaling has no universal direction of bias when two scale mixtures are compared.
A Bonferroni familywise bound requires valid component tests and allocation across the actual
tested family.

\paragraph{Measurement floor.}
Correction acts on the body, while measurement noise is added through the observing process.
A fixed additive noise floor can reduce sensitivity without changing the population difference
between two responses. To isolate this distinction in a direct observation channel, let
$\mathbf u_t^{\mathrm{obs}}=D\bm\varepsilon_t+\mathbf m_t$, with predictable $D$ and centered measurement
noise with conditional covariance $R_t$, conditionally uncorrelated with the input. Then
\[
H_t^{\mathrm{obs}}=D(1-K_t)^2\Sigma_tD'+R_t.
\]
With $D$, $R$, and $\Sigma$ fixed across the same-history comparison, the trace response is
$-K(h)(2-K(h))\operatorname{tr}(D\Sigma D')$.
The identity with $D=I$ applies only to a direct measurement in the input coordinates.
Actual position or angular astrometry requires propagation through the dynamic and observation
models. Signal-to-noise ratio affects power together with sample size, cadence, dependence, and
parameter uncertainty; physical innovation variance below per-epoch measurement variance is not an absolute
identifiability threshold.

Reference fitting and rejection rules can remove or alter components of the signal.
Their effect depends on the fitted subspace, sampling, and selection rules.
A held-out event window and a modeled error law
can help assess this, but the actual reduction must be included in simulation.
Subtracting an estimated measurement covariance does not automatically produce an unbiased
or positive-definite physical covariance.

\paragraph{Observable subspace.}
Optical astrometry supplies two sky-plane angles per epoch \citep{carpino2003errors,veres2017statistical}. Convert between angles and
physical transverse displacement through a declared geometry and range model.
The line-of-sight component is model-inferred unless additional measurements constrain it.
Trace statistics must use the measured subspace or explicitly account for uncertainty in inferred
components. Radar range and range-rate measurements add information through different measurement maps
\citep{porcelli2022radar}; observability still depends on the viewing geometry. A correction can
be weak or invisible in the measured directions. Corollary~\ref{cor:joint}(iii) therefore uses
separation of observation-level signatures: the instrument must retain the difference between
mechanisms, and calibration must distinguish it from measurement and selection effects.

\clearpage
\section{Finite-sample behavior under the illustrative model}
\label{app:mc}

The population contrasts separate specified physical laws, but a finite trajectory may resolve
those differences poorly. The Monte Carlo experiments examine what happens when the environment
changes, an event must be located from the record, occasional large disturbances affect estimated
moments, or measurement noise obscures a variance response. The kurtosis experiment uses a
Student-$t$ disturbance; measurement-noise experiments add Gaussian error. Correction follows
the linear mechanisms of Figure~\ref{fig:regimes} under the separate forcing and noise laws
specified here. Rejection frequencies describe performance under these illustrative laws;
an observing program requires its own calibration and resampling justification
(Appendix~\ref{app:empirical}).

The within-object variance-response statistic is the mean squared innovation over the
$H=\mcHorizon$ post-onset steps, divided by the estimated quiet variance, minus one.
The restricted bootstrap standardizes innovations under the null second-moment law and resamples
them. Each resample recomputes the statistic; the full-pipeline variant also repeats event-time
selection. The resulting distribution supplies a one-sided critical value. First-order validity
requires reproduction of the null's conditional second-moment structure under the dependence
and fourth-moment conditions discussed in Appendix~\ref{app:empirical}.

Table~\ref{tab:mc} reports rejection frequencies and shape estimates for these experiments.
Under the flat inert null, the restricted-bootstrap rejection frequencies are
\mcSizeShort{} at $T=\mcTshort$, \mcSizeMed{} at $T=\mcTmed$, and
\mcSizeLong{} at $T=\mcTlong$.

\begingroup
\small\singlespacing
\setlength{\tabcolsep}{4pt}
\setlength{\LTpre}{6pt}
\setlength{\LTpost}{6pt}
\begin{longtable}{@{}>{\raggedright\arraybackslash}p{\dimexpr.61\linewidth-6\tabcolsep\relax}*{3}{>{\centering\arraybackslash}p{.13\linewidth}}@{}}
	\caption{\textbf{Monte Carlo behavior of the diagnostics.} Rejection frequencies and shape estimates under the specified illustrative models.}\label{tab:mc}\\
	\toprule
	Quantity (nominal \mcAlpha\% where applicable) & $T=\mcTshort$ & $T=\mcTmed$ & $T=\mcTlong$ \\
	\midrule
	\endfirsthead
	\multicolumn{4}{@{}l}{\textbf{Table~\thetable{} (continued).}}\\
	\toprule
	Quantity (nominal \mcAlpha\% where applicable) & $T=\mcTshort$ & $T=\mcTmed$ & $T=\mcTlong$ \\
	\midrule
	\endhead
	\midrule
	\multicolumn{4}{r@{}}{\footnotesize Continued on the next page.}\\
	\endfoot
	\bottomrule
	\endlastfoot
	Size of restricted bootstrap $v_h$ test under (I), flat environment, known $t_0$ & \mcSizeShort & \mcSizeMed & \mcSizeLong \\
	Size under (I), $t_0$ located by an innovations mean-shift detector, event time treated as fixed during resampling & \mcDetNaiveShort & \mcDetNaiveMed & \mcDetNaiveLong \\
	Same detected $t_0$, detector re-run inside each bootstrap resample (full pipeline) & \mcDetFullShort & \mcDetFullMed & \mcDetFullLong \\
	Size under (I) with a $\mcGrad\times$ declining $g(W_t)$ across the arc, naive test & \mcGradNaiveShort & \mcGradNaiveMed & \mcGradNaiveLong \\
	Same gradient null, $g(W_t)$-standardized innovations & \mcGradStdShort & \mcGradStdMed & \mcGradStdLong \\
	Same gradient null, standardized by an \emph{estimated} $\hat g(W_t)$ (feasible) & \mcGradFitShort & \mcGradFitMed & \mcGradFitLong \\
	Power of $v_h$ test under (S), $K_{\max}=\mcKhi$ & \mcPowHiShort & \mcPowHiMed & \mcPowHiLong \\
	Power of $v_h$ test under (S), $K_{\max}=\mcKlo$ & \mcPowLoShort & \mcPowLoMed & \mcPowLoLong \\
	Power at $K_{\max}=\mcKhi$ with measurement noise, $\operatorname{tr}R/\operatorname{tr}\Sigma=\mcSnrRatioLo$ & \mcPowSnrOneShort & \mcPowSnrOneMed & \mcPowSnrOneLong \\
	Same, $\operatorname{tr}R/\operatorname{tr}\Sigma=\mcSnrRatioHi$ & \mcPowSnrThreeShort & \mcPowSnrThreeMed & \mcPowSnrThreeLong \\
	Cross-sectional variance-\emph{level} test size under (I) vs a validated ballistic cohort (\mcXsecN{} refs) & \mcXsecSizeShort & \mcXsecSizeMed & \mcXsecSizeLong \\
	Cross-sectional power, constant gain $K=\mcKhi$, $\operatorname{tr}R/\operatorname{tr}\Sigma=\mcSnrRatioLo$ & \mcXsecPowHiOneShort & \mcXsecPowHiOneMed & \mcXsecPowHiOneLong \\
	Same, $\operatorname{tr}R/\operatorname{tr}\Sigma=\mcSnrRatioHi$ & \mcXsecPowHiThreeShort & \mcXsecPowHiThreeMed & \mcXsecPowHiThreeLong \\
	Cross-sectional power, constant gain $K=\mcKlo$, $\operatorname{tr}R/\operatorname{tr}\Sigma=\mcSnrRatioLo$ & \mcXsecPowLoOneShort & \mcXsecPowLoOneMed & \mcXsecPowLoOneLong \\
	Same, $\operatorname{tr}R/\operatorname{tr}\Sigma=\mcSnrRatioHi$ & \mcXsecPowLoThreeShort & \mcXsecPowLoThreeMed & \mcXsecPowLoThreeLong \\
	DF rejection rate with additive outliers (illustrative test distortion) & \mcLeakShort & \mcLeakMed & \mcLeakLong \\
	Same experiment, with an additive-outlier removal pre-step & \mcLeakRobShort & \mcLeakRobMed & \mcLeakRobLong \\
	DF rejection rate, clean random walk (size check) & \mcCleanShort & \mcCleanMed & \mcCleanLong \\
	Displacement level-shift test: false-positive rate under (I) & \mcDispSizeShort & \mcDispSizeMed & \mcDispSizeLong \\
	$\hat\kappa^{(S)}-\hat\kappa^{\mathrm{twin}}$, conditionally standardized, $\lambda=\mcLambda$ (O3 excess-kurtosis reduction; MC s.e.) & $\mcKurtShort$ (\mcKurtSEShort) & $\mcKurtMed$ (\mcKurtSEMed) & $\mcKurtLong$ (\mcKurtSELong) \\
	Fraction of arcs with an estimated positive kurtosis gap, $\lambda=\mcLambda$ & \mcKurtPosShort & \mcKurtPosMed & \mcKurtPosLong \\
	$\hat\kappa^{(S)}-\hat\kappa^{\mathrm{twin}}$ at the lower share $\lambda=\mcLambdaLo$ & $\mcKurtLamShort$ & $\mcKurtLamMed$ & $\mcKurtLamLong$ \\
	Pooled-gain gap $\hat\kappa((1-K_t)r_t)-\hat\kappa(r_t)$, same mixed input $r_t$ & $\mcWCShort$ & $\mcWCMed$ & $\mcWCLong$ \\
\end{longtable}
\footnotesize
\noindent\textit{Notes.} \mcNrep{} fixed-seed replications per cell. The one-sided
test imposes the null during resampling and compares the mean squared innovation over
$H=\mcHorizon$ post-onset steps with the estimated quiet variance; detected-$t_0$ rows use a cumulative-sum (CUSUM) mean-shift detector. Gradient rows decline by a factor
\mcGrad{} and use either the true $g(W_t)$ or an $\hat g(W_t)$ fitted from the arc.
The displacement row flags a standardized absolute CUSUM statistic above the upper
$1-\mcAlpha/100$ bootstrap quantile, reselecting the event time in each whole-series resample.
The outlier rows add \mcOutliers{} outliers and use a demeaned Dickey--Fuller statistic.
Negative $\hat\kappa$ gaps denote estimated excess-kurtosis reduction; parentheses report Monte Carlo standard errors. The event disturbance is
standardized Student-$t$ with \mcTailDf{} degrees of freedom, with
$\sigma_\nu/\sigma_\zeta=\mcSdNu$ and default $\lambda=\mcLambda$. Measurement-noise rows add Gaussian error at
the stated variance ratio. The cross-sectional rows compare a candidate's log sample variance against a cohort of
\mcXsecN{} validated ballistic references by a one-sided \mcAlpha\% prediction bound; the size row draws the
candidate from the ballistic law, the power rows from a constant-gain body of observed level
$(1-K)^2+\operatorname{tr}R/\operatorname{tr}\Sigma$.
\par
\endgroup

Power at $K_{\max}=\mcKhi$ runs from \mcPowHiShort{} to \mcPowHiLong{};
the statistic keeps a fixed $H=\mcHorizon$-step response horizon as the total arc grows. Because the short-arc size sits above nominal
(\mcSizeShort{} at $T=\mcTshort$), short-arc power is reported raw rather than size-adjusted and
must be read alongside that size; the longer-arc sizes are near nominal.

The paired detected-$t_0$ rows compare resampling with the detected time held fixed
against resampling that repeats the detector. The gradient rows impose a declining covariance law: their naive rejection
frequencies span \mcGradNaiveRange{}, while adjustment using the stipulated $g(W_t)$ brings
them near the flat-null frequencies.

The additive-outlier experiment illustrates distortion of a Dickey--Fuller statistic:
rejection frequencies for the contaminated random walk are \mcLeakShort{}, \mcLeakMed{},
and \mcLeakLong{} at the short, medium, and long arcs. The corresponding clean-model
frequencies are \mcCleanShort{}, \mcCleanMed{}, and \mcCleanLong{}.
These frequencies concern the specified random-walk and outlier models.

For the specified kurtosis experiment, the measured stabilized-minus-twin gaps are reported with
Monte Carlo standard errors. The single-arc positive-gap frequencies are \mcKurtPosShort{}
at the shortest arc and \mcKurtPosLong{} at the longest.
The pooled-gain row compares $\hat\kappa((1-K_t)r_t)$ with $\hat\kappa(r_t)$ for the same
mixed input before recovery scaling $r_t=(1-\lambda)\zeta_t+\lambda\nu_t$. It isolates gain-induced
scale mixing with a common input shape. Proposition~\ref{prop:o3inv}(i) compares such a
scale mixture with a constant-scale law; it does not order two arbitrary mixtures.

The measurement-noise rows use the direct additive observation channel specified in
Appendix~\ref{app:empirical}. Added Gaussian noise reduces power under that model, bringing
it toward size for the larger floor. This illustrates a loss of sensitivity at the stated
sample sizes; it does not create an absolute per-epoch variance threshold for all observations.

Constant recovery gain gives $v^{S}_h\equiv0$ when applied in both continuations
(Proposition~\ref{prop:o4}(ii)), so this case is evaluated cross-sectionally. In the simulation,
$K>0$ holds the innovation covariance at $(1-K)^2\Sigma$, below the common ballistic
baseline $\Sigma$. The test flags a candidate whose log sample variance falls below the
one-sided \mcAlpha\% prediction bound from \mcXsecN{} validated ballistic references.
Its null rejection frequencies are \mcXsecSizeShort{}, \mcXsecSizeMed{}, and
\mcXsecSizeLong{} across the three arc lengths.

At the moderate measurement floor $\operatorname{tr}R/\operatorname{tr}\Sigma=\mcSnrRatioLo$,
power ranges from \mcXsecPowLoOneShort{} to \mcXsecPowLoOneLong{} at $K=\mcKlo$,
and from \mcXsecPowHiOneShort{} to \mcXsecPowHiOneLong{} at $K=\mcKhi$.
At the higher floor $\operatorname{tr}R/\operatorname{tr}\Sigma=\mcSnrRatioHi$, power
increases with arc length from \mcXsecPowHiThreeShort{} to \mcXsecPowHiThreeLong{} at
$K=\mcKhi$, and from \mcXsecPowLoThreeShort{} to \mcXsecPowLoThreeLong{} at $K=\mcKlo$.
The comparison imposes a common baseline $\Sigma$ for candidate and reference cohort.
Its empirical size and power therefore characterize that matched simulation.
A field cohort needs a separate argument for common coupling, environment and observation
laws; resolving a level difference alone does not identify a constant-gain controller.

\subsection{Sensitivity of forecast-residual regularity tests}

A KPSS-type statistic tests a stationarity null under its own dependence and moment conditions
\citep{kwiatkowski1992kpss}; a Dickey--Fuller statistic tests a scalar $+1$ unit root.
Neither non-rejection establishes a physical bound, and neither directly tests the entire
unit-modulus boundary of a multivariate lag-one projection.
Efficient variants and lag selection \citep{elliott1996ers,ng2001lag} can improve performance
in their specified models, without universal finite-sample dominance.

Additive outliers can distort unit-root rejection \citep{franses1994outliers}.
The removal experiment reduces the rejection frequencies to \mcLeakRobShort{},
\mcLeakRobMed{}, and \mcLeakRobLong{} across the three arc lengths, still above the
nominal \mcAlpha\%.
That preprocessing does not eliminate all stationary pulse-driven alternatives.
Wild-bootstrap unit-root methods can address particular changing-variance models
\citep{cavaliere2008bootstrap}; their assumptions and target statistic must be checked rather
than transferred to every variance-response test. Cadence and memory sensitivity likewise
require the qualifications in Remark~\ref{rem:frequency}.

\bibliographystyle{unsrtnat}
\bibliography{inert_refs}

\end{document}